\documentclass[10pt]{article}
\usepackage{amsfonts}
\usepackage{amsmath}
\usepackage{mathrsfs}
\usepackage{color}
\usepackage{tikz}
\usepackage{mathrsfs,amscd,amssymb,amsthm,amsmath,bm,graphicx,psfrag,subfigure,url}
\usepackage[titletoc]{appendix}

\usepackage{indentfirst}

\def \[{\begin{equation}}
\def \]{\end{equation}}

\newtheorem{thm}{Theorem}[section]

\newtheorem{defi}{Definition}
\newtheorem{claim}{Claim}
\newtheorem{Case}{Case}
\newtheorem{subc}{Subcase}
\newtheorem{fact}{Fact}
\newtheorem{lem}[thm]{Lemma}
\newtheorem{cor}[thm]{Corollary}
\newtheorem{ex}[thm]{Example}

\newtheorem{Obs}{Observation}

\newenvironment{wst}
{\setlength{\leftmargini}{1.5\parindent}
 \begin{itemize}
 \setlength{\itemsep}{-1.1mm}}
{\end{itemize}}

\begin{document}
\begin{center}{\Large \bf Mixed graphs with cut vertices having exactly two positive eigenvalues}
\vspace{4mm}

{\large Xiaocong He, \ \ Lihua Feng\footnote{Corresponding author.  \\ \hspace*{5mm}{\it Email addresses}: hexc2018@qq.com (X.C. He),\ \ fenglh@163.com (L.H. Feng)}}
\vspace{3mm}

School of Mathematics and Statistics, HNP-LAMA, Central South University, Changsha, Hunan 410083, P. R. China

\end{center}


\noindent {\bf Abstract}: \
A mixed graph is obtained by orienting some edges of a simple graph. The positive inertia index of a mixed graph is defined as the number of positive eigenvalues of its Hermitian adjacency matrix, including multiplicities. This matrix was introduced by Liu and Li, independently by Guo and Mohar, in the study of graph energy. Recently, Yuan et al. characterized the mixed graphs with exactly one positive eigenvalue. In this paper, we study the positive inertia indices of mixed graphs and  characterize the mixed graphs with cut vertices having positive inertia index 2.

\vspace{2mm} \noindent{\bf Keywords}: Mixed graph; Cut vertices; Hermitian adjacency matrix; Positive inertia index
\vspace{2mm}

\setcounter{section}{0}
\section{\normalsize Introduction}\setcounter{equation}{0}
We will start with introducing some background that will lead to our main results. Some significant previously established facts will also be presented.
\subsection{\normalsize Background}
Let $G$ be a simple graph with vertex set $V(G)=\{v_1,v_2,\ldots,v_n\}$ and edge set $E(G)$. A \textit{mixed graph} $\widetilde{G}$ is obtained by orienting some edges of $G$. Accordingly, $G$ is referred to as the underlying graph of $\widetilde{G}$. Consequently, a mixed graph $\widetilde{G}$ is an ordered triple $(V(G), E_0, E_1)$, where $E_0$ is the undirected edge set and $E_1$ is the directed edge (arc) set. It is obvious that $\widetilde{G}$ is undirected if $E_0=E(G)$ while $\widetilde{G}$ is an oriented graph if $E_0=\emptyset$. Thus, mixed graphs are the generalizations of simple graphs and oriented graphs.

The $Hermitian$ $adjacency$ $matrix$ $H(\widetilde{G})=(h_{st})_{n\times n}$ of $\widetilde{G}$ was proposed, independently, by Liu and Li \cite{lxl}, and Guo and Mohar \cite{guo}. It is defined as
\begin{center}
$h_{st}=\left\{
  \begin{array}{ll}
    1, & \hbox{if $v_sv_t$ is an undirected edge;} \\ [1mm]
    i, & \hbox{if $\overrightarrow{v_sv_t}$ is an arc;} \\ [1mm]
    -i, & \hbox{if $\overrightarrow{v_tv_s}$ is an arc;} \\ [1mm]
    0, & \hbox{otherwise,}
  \end{array}
\right.$
\end{center}
where $i=\sqrt{-1}$. Note that $H(\widetilde{G})$ is Hermitian, that is, $H(\widetilde{G})^{*}=H(\widetilde{G})$, where $H(\widetilde{G})^{*}$ denotes the
conjugate transpose of $H(\widetilde{G})$. Thus the eigenvalues  of $H(\widetilde{G})$ are real,  called the $H$-eigenvalues  of $\widetilde{G}$. The $H$-rank of $\widetilde{G}$, denoted by $rk(\widetilde{G})$, is the rank of $H(\widetilde{G})$, and the number of positive, negative and zero eigenvalues of $H(\widetilde{G})$ are defined as $positive$ $inertia$ $index$, $negative$ $inertia$ $index$ and $nullity$ of $\widetilde{G}$, denoted by $p(\widetilde{G})$, $n(\widetilde{G})$ and $\eta(\widetilde{G})$, respectively. Obviously, $p(\widetilde{G})+n(\widetilde{G})=rk(\widetilde{G})$. The \textit{inertia index} of $\widetilde{G}$ is the triple ${\rm In}(\widetilde{G})=(p(\widetilde{G}), n(\widetilde{G}), \eta(\widetilde{G}))$. Clearly, the Hermitian adjacency matrix of a mixed graph is a natural generalization of the adjacency matrix of a simple graph. Therefore the mixed graph attracts more and more researchers' attention; see \cite{guo,hu,lscc,ly,lxl,mo,mo20,wiss,WW,34,radii,yuan}.

In \cite{guo,mo}, the authors considered the so-called \textit{four-way} switching operation keeping the $H$-spectrum of a mixed graph unchanged.  We describe the definition in the matrix-theoretic language below.
\begin{defi}\cite{mo}
Let $\widetilde{G_1}$ and $\widetilde{G_2}$ be two mixed graphs. A four-way switching (FWS for short) is the operation of changing
$\widetilde{G_1}$ into $\widetilde{G_2}$, if there exists a diagonal matrix $Q$ with $Q_{vv}\in\{\pm1, \pm i\}$ such that $H(\widetilde{G_2})=Q^{-1}H(\widetilde{G_1})Q$.
\end{defi}

In fact, it is often difficult to check whether a mixed graph $\widetilde{G_2}$ can be obtained from a mixed graph $\widetilde{G_1}$ by a four-way switching. Because of this, Mohar \cite{mo} gave some special cases of four-way switching called \textit{two-way switching}, which can be described
as a special similarity transformation. There are really two different cases as follows.

\begin{wst}
\item[{\rm (1)}] \textit{Directed two-way switching}. Reversing all edges in an edge cut consisting of
      directed edges only.
\item[{\rm (2)}] \textit{Mixed two-way switching}. If there is an edge cut which contains undirected
      edges (possibly empty) and directed edges (possibly empty) in one direction
      only, replacing the each directed edge in the cut by a single undirected edge,
      and replacing each former undirected edge by a single directed edge in the
      direction opposite to the direction of former directed edges.
\end{wst}

The converse ${\widetilde{G}}^{\top}$ of a mixed graph $\widetilde{G}$ is obtained from $\widetilde{G}$ by reversing all directed edges in $\widetilde{G}$. They are cospectral since $H(\widetilde{G}^{\top})=H(\widetilde{G})^{\top}$. Two mixed graphs $\widetilde{G_1}$ and $\widetilde{G_2}$ are said to be switching equivalent if $\widetilde{G_1}$ can be changed into $\widetilde{G_2}$ by a FWS and taking a possible converse.

The circle group, which is denoted by $\mathbb{T}=\{z\in \mathbb{C} : |z|=1\}$, is a subgroup of the multiplicative group of all nonzero complex numbers $\mathbb{C}^{\times}$. If $v_sv_t\in E(G)$, we denote by $e_{v_sv_t}$ the oriented edge from $v_s$ to $v_t$. Let $\overrightarrow{E}$ be the set of oriented edges of $G$, i.e., $\overrightarrow{E}=\{e_{v_sv_t}, e_{v_tv_s}: v_sv_t\in E(G)\}$. A \textit{complex unit gain graph} $\Phi=(G,\mathbb{T},\varphi)$ is a graph with a gain function $\varphi :\overrightarrow{E} \rightarrow \mathbb{T}$, such that $\varphi(e_{v_sv_t})=\varphi(e_{v_tv_s})^{-1}=\overline{\varphi(e_{v_tv_s})}$ for any $v_sv_t\in E(G)$. Clearly, the complex unit gain graph is a natural generalization of a simple graph. So, the complex unit gain graph attracts much plenty of researchers' attention; see \cite{hsj,lscc,15,16,luy,lu,27,xu,38,hxc}.

We usually write $G^{\varphi}$ for a complex unit gain graph $\Phi=(G,\mathbb{T},\varphi)$. The adjacency matrix of $G^{\varphi}$ is the   Hermitian matrix  $A(G^{\varphi})=(a_{st}^\varphi)_{n\times n}$, where
\begin{center}
$a_{st}^\varphi=\left\{
  \begin{array}{ll}
    \varphi(e_{v_sv_t}), & \hbox{if $v_sv_t\in E(G)$;} \\
    0, & \hbox{otherwise.}
  \end{array}
\right.$
\end{center}
The $rank$ of $G^{\varphi}$, denoted by $rk(G^{\varphi})$, is the rank of $A(G^{\varphi})$. The $positive$ $inertia$ $index$ $p(G^{\varphi})$, the $negative$ $inertia$ $index$ $n(G^{\varphi})$ and the $nullity$ $\eta(G^{\varphi})$ of $G^{\varphi}$ are defined to be the number of positive eigenvalues, negative eigenvalues and zero eigenvalues of $A(G^{\varphi})$, respectively. Obviously, $p(G^{\varphi})+n(G^{\varphi})=rk(G^{\varphi})$.

If $\varphi(\overrightarrow{E})\subseteq\{1, \pm i\}\subseteq\mathbb{T}$, then for any $v_sv_t\in E(G)$, we define
\begin{center}
$h_{st}=\left\{
  \begin{array}{ll}
    \varphi(e_{v_sv_t}), & \hbox{if $v_sv_t\in E(G)$;} \\
    0, & \hbox{otherwise.}
  \end{array}
\right.$
\end{center}
Then $(h_{st})_{n\times n}$ is the Hermitian adjacency matrix of a mixed graph $\widetilde{G}$ and we have $H(\widetilde{G})=(h_{st})_{n\times n}=A(G^{\varphi})$. Hence, the complex unit gain graphs are the generalizations of mixed graphs. We can view a mixed graph as a special case of a complex unit gain graph.

Two complex unit gain graphs $G^{\varphi}$ and $G^{\varphi'}$ are {\it switching equivalent} if there is a mapping $\theta: V(G) \rightarrow \mathbb{T}$ such
that $\varphi'(uv)=\theta(u)^{-1}\varphi(uv)\theta(v)$. In this case, the mapping $\varphi'$ can be written as $\varphi\theta$. It leads to that
$$A(G^{\varphi'})=diag(\theta(v_1), \theta(v_2),\ldots, \theta(v_n))^{-1}A(G^{\varphi})diag(\theta(v_1), \theta(v_2),\ldots, \theta(v_n)).$$
Therefore, $G^{\varphi}$ and $G^{\varphi'}$ share the same spectrum. It is clear that the switching equivalence is an equivalence relation.

We write $C_n$, $S_n$, $K_n$ for the cycle, star and complete graph of order $n$, respectively. A mixed graph is called a \textit{mixed cycle, mixed star, mixed complete graph}, etc., if its underlying graph is a cycle, star, complete graph, etc., respectively. A \textit{cut vertex} in $G$ is a vertex whose removal increases the number of connected components of $G$. For $v\in V$, $v$ is called a $pendant$ $vertex$ if its degree $d_G(v)=1$. Let $K_{n_1, n_2,\ldots, n_k}$ denote the complete $k$-partite graph with $V_1, V_2,\ldots, V_k$ being the partition class such that $|V_j|=n_j$ for $1\leq j\leq k$. We follow \cite{1} for notations and terminologies not defined here.

A \textit{subgraph} of $\widetilde{G}$ (resp. $G^{\varphi}$) is a subgraph of $G$ in which each edge preserves the original direction (resp. gain fuction) in $\widetilde{G}$ (resp. $G^{\varphi}$). Let $v\in V(G)$, we write $\widetilde{G}-v$ (resp. $G^{\varphi}-v$) for the \textit{induced subgraph} obtained from $\widetilde{G}$ (resp. $G^{\varphi}$) by deleting the vertex $v$ and all edges incident with $v$. For an induced subgraph $\widetilde{R}$ (resp. $R^{\varphi}$) of $\widetilde{G}$ (resp. $G^{\varphi}$), denote by $\widetilde{G}-\widetilde{R}$ (resp. $G^{\varphi}-R^{\varphi}$), the subgraph obtained from $\widetilde{G}$ (resp. $G^{\varphi}$) by deleting all vertices of $\widetilde{R}$ (resp. $R^{\varphi}$) and all incident edges. For a vertex $v$ of $V(G)\backslash V(R)$, we write $\widetilde{R}+v$ (resp. $R^{\varphi}+v$) to denote the induced subgraph of $\widetilde{G}$ (resp. $G^{\varphi}$) with vertex set $V(R)\bigcup\{v\}$. A vertex of $\widetilde{G}$ (resp. $G^{\varphi}$) is called a \textit{pendant vertex} if it is of degree one in $G$. A vertex of $\widetilde{G}$ (resp. $G^{\varphi}$) is called a \textit{cut vertex} if it is a cut vertex of $G$.

The spectral-based graph invariants are widely investigated in the literature, such as inertia index \cite{10,g1,g2,14,23,26,o,wiss,smith,33,36,37,hxc} of simple graphs, signed graphs, mixed graphs and complex unit gain graphs, nullity \cite{3,5,7,41,42,11,22,25,28,36,37,39} of simple graphs and signed graphs, rank \cite{2,9,29,30,32,35} of simple graphs, signed graphs and complex unit gain graphs, the $H$-rank \cite{6,8,34} of mixed graphs and skew-rank \cite{12,13,17,18,19,20,24,31,40} of oriented graphs. The spectral parameter ``inertia index'' can determine the structure of a graph to some extent and
has attracted much attention recently. Smith \cite{smith} characterized all (undirected) graphs with positive inertia index 1. Fan et al. \cite{10} determined sharp bounds for the positive and negative inertia index of a graph. Oboudi \cite{o} completely characterized the (undirected) graphs with exactly two non-negative eigenvalues. Yu et al. \cite{36} determined the signed graphs with positive inertia index 1 and the signed graphs containing pendant vertices with positive inertia index 2. Wang et al. \cite{33} extended the above work to the signed graphs containing cut vertices with positive inertia index 2. Gregory et al. \cite{g2} investigated the subadditivity of the inertia indices and gave some properties of Hermitian rank that can determine the biclique decomposition number. Gregory et al. \cite{g1} studied the inertia indices of a partial join of two graphs and provided certain connections between biclique decompositions of partial joins of graphs and the inertia indices. Geng et al. \cite{geng} studied characterizations of graphs with given inertia index achieving the maximum diameter.

Recently, much attention has been paid to the inertia index of mixed graphs. Wissing and van Dam \cite{wiss} characterized all mixed graphs with negative inertia index 1. He et al. \cite{hao} studied the negative and positive inertia indices of mixed unicyclic graphs and gave the upper and lower bounds of the negative and positive inertia indices for mixed graphs. Zheng et al. \cite{zheng} gave some eigenvalue inequalities between two mixed graphs by means of inertia indices. Wei et al. \cite{WW} investigated relations between the inertia indices of a mixed graph and those of its underlying graph. Yuan et al. \cite{yuan} provided a characterization of mixed graphs with positive inertia index 1 and studied the determination  of some mixed graphs  by their $H$-spectra.

{Motivated by the results in \cite{guo,lxl,mo,WW,wiss,33,34,36,37,yuan}, in this paper, we  consider the mixed graphs with connected underlying graph and characterize all mixed graphs with cut vertices which have positive inertia index 2. For more general setting, we leave it for further research.

\subsection{\normalsize Main results}
Before announcing the main theorem, we introduce one class mixed graph. Let $\overrightarrow{C_3}(t_1, t_2, t_3)$ be the mixed complete tripartite graph shown in Figure. 1, where $|A|=t_1, |B|=t_2, |C|=t_3$. An \textit{odd triangle} is a mixed triangle containing odd number of directed edges. In particular, $\overrightarrow{C_3}(1, 1, 1)$ is exactly an odd triangle. Similarly, a mixed triangle containing even number of directed edges is called an \textit{even triangle}.
\begin{figure}[h!]
\begin{center}
\psfrag{1}{$A$}\psfrag{2}{$B$}\psfrag{3}{$C$}\psfrag{4}{Figure 1: $\overrightarrow{C_3}(t_1, t_2, t_3)$}
\includegraphics[width=40mm]{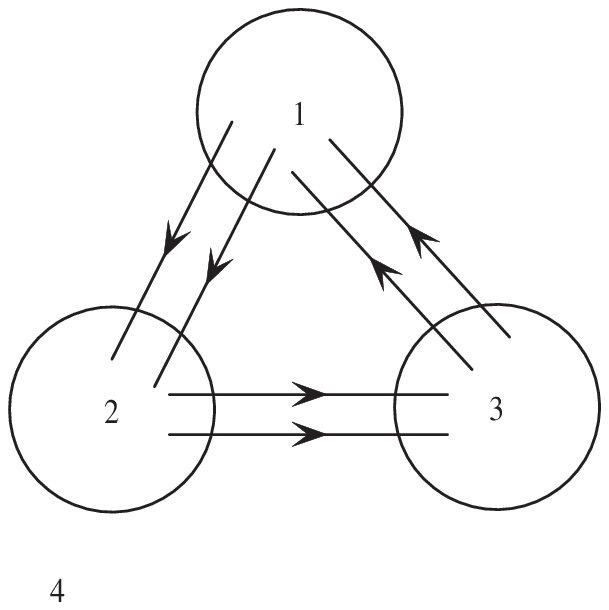} \\
\end{center}
\end{figure}

We now list our main results in the following.
\begin{thm}\label{thm1.1}
Let $\widetilde{G}$ be a connected mixed graph with pendant vertices. Then $p(\widetilde{G})=2$ if and only if $\widetilde{G}$ is obtained by adding some undirected edges (possibly empty) and directed edges (possibly empty) with arbitrary direction between the center of a mixed star and some vertices of a mixed graph $\widetilde{F}$, where $\widetilde{F}$ is switching equivalent   to a complete multipartite graph or to some $\overrightarrow{C_3}(t_1, t_2, t_3)$.
\end{thm}

A most interesting problem attracting us is to determine all mixed graphs without pendant vertices whose positive inertia indices are 2. It seems difficult for us to give a complete solution now. Note that a mixed graph with a pendant vertex must have a cut vertex. In the following, we extend the result of Theorem \ref{thm1.1} to   consider the mixed graphs with cut vertices which have positive inertia index 2. In order to accomplish our goal,  we should have the aid of the following graph transformations and terminologies.

Let $\widetilde{G_1}$ and $\widetilde{G_2}$ be two mixed graphs. Let $\widetilde{G_1}\bullet v\bullet\widetilde{G_2}$ be the mixed graph obtained by identifying one vertex $v'$ of $\widetilde{G_1}$ with one vertex $v''$ of $\widetilde{G_2}$ (we denote the new vertex by $v$).

Let $K(q_1,\ldots, q_r; n_1, \ldots, n_k; p)$ be a simple graph obtained from $K_{q_1,\ldots, q_r}$ and $K_{n_1, \ldots, n_k}$ by adding a new vertex $v$ and adding edges such that $v$ is adjacent to all vertices of $K_{q_1,\ldots, q_r}$ and adjacent to all vertices of $\bigcup_{j=1}^{p} V_j$, where $1\leq p\leq k$ and $V_1, V_2,\ldots, V_k$ are the partition classes of $V(K_{n_1, n_2,\ldots, n_k})$ such that $|V_j|=n_j$ for $1\leq j\leq k$ (see Figure. 2). Similarly, let $K(0; n_1, \ldots, n_k; p)$ be a simple graph obtained from $K_{n_1, \ldots, n_k}$ by adding a new vertex $v$ and adding edges such that $v$ is adjacent to all vertices of $\bigcup_{j=1}^{p} V_j$, where $1\leq p\leq k$. If $K_{q_1,\ldots, q_r}$ and $K_{n_1, \ldots, n_k}$ are complete graphs $K_r$ and $K_k$, respectively, then we denote $K(q_1,\ldots, q_r; n_1, \ldots, n_k; p)$ by $K(r; k; p)$. Similarly, if $K_{n_1, \ldots, n_k}$ is a complete graph $K_k$, then we denote $K(0; n_1, \ldots, n_k; p)$ by $K(0; k; p)$.
\begin{figure}[h!]
\begin{center}
\psfrag{v}{$v$}\psfrag{1}{Figure 2: $K(3, 2; 2, 2, 3; 2)$}
\includegraphics[width=70mm]{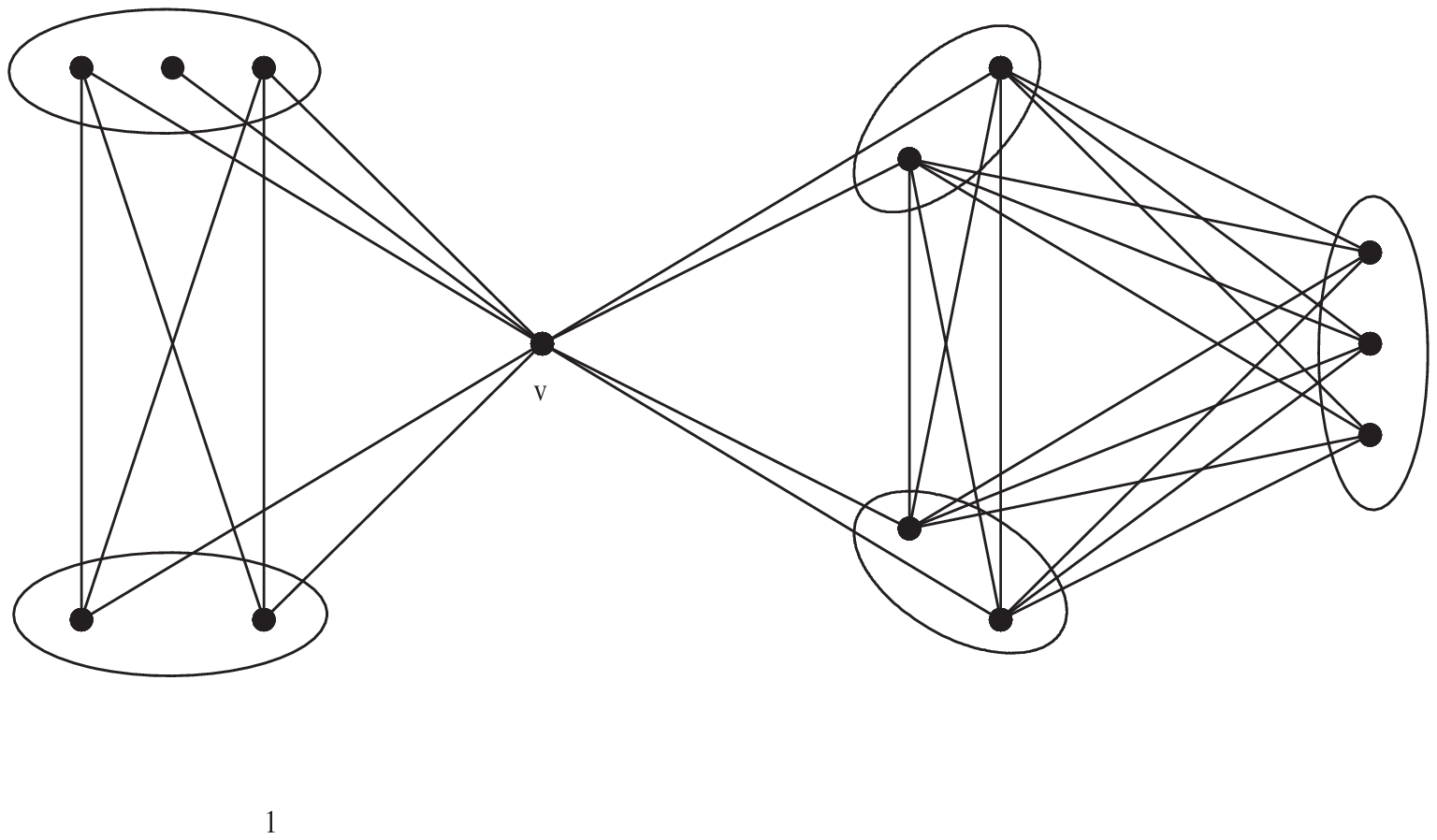} \\
\end{center}
\end{figure}

For nonnegative integers $a, b, c$ and $d$ satisfying $1\leq a+b+c+d\leq k$, let
$$K(q_1,\ldots, q_r; n_1, \ldots, n_k; a^i, b^{-i}, c^{1}, d^{-1}, k-a-b-c-d)=(K(q_1,\ldots, q_r; n_1, \ldots, n_k; a+b+c+d), \mathbb{T}, \varphi)$$
be a complex unit gain graph with underlying graph $K(q_1,\ldots, q_r; n_1, \ldots, n_k; a+b+c+d)$ and gain function $\varphi$:
\begin{center}
$\left\{
  \begin{array}{ll}
    \varphi(e_{vu})=\varphi(e_{uv})^{-1}=i, & \hbox{if $u\in \bigcup_{j=1}^{a}V_j$;} \\  [1mm]
    \varphi(e_{vu})=\varphi(e_{uv})^{-1}=-i, & \hbox{if $u\in \bigcup_{j=a+1}^{a+b}V_j$;} \\ [1mm]
    \varphi(e_{vu})=\varphi(e_{uv})^{-1}=1, & \hbox{if $u\in \bigcup_{j=a+b+1}^{a+b+c}V_j$;} \\ [1mm]
    \varphi(e_{vu})=\varphi(e_{uv})^{-1}=-1, & \hbox{if $u\in \bigcup_{j=a+b+c+1}^{a+b+c+d}V_j$;} \\ [1mm]
    \varphi(e_{uw})=\varphi(e_{wu})^{-1}=1, & \hbox{if $uw$ is any other edge of $K(q_1,\ldots, q_r; n_1, \ldots, n_k; a+b+c+d)${.}}
  \end{array}
\right.$
\end{center}
Similarly, we can define $K(0; n_1, \ldots, n_k; a^i, b^{-i}, c^{1}, d^{-1}, k-a-b-c-d)$. If $K(q_1,\ldots, q_r; n_1, \ldots, n_k; a+b+c+d)$ is $K(r; k; a+b+c+d)$, then we denote
$$K(q_1,\ldots, q_r; n_1, \ldots, n_k; a^i, b^{-i}, c^{1}, d^{-1}, k-a-b-c-d)$$
by $K(r; k; a^i, b^{-i}, c^{1}, d^{-1}, k-a-b-c-d)$. Similarly, if $K(0; n_1, \ldots, n_k; a+b+c+d)$ is $K(0; k; a+b+c+d)$, then we denote
$K(0; n_1, \ldots, n_k; a^i, b^{-i}, c^{1}, d^{-1}, k-a-b-c-d)$
by $K(0; k; a^i, b^{-i}, c^{1}, d^{-1}, k-a-b-c-d)$.

Let $\widetilde{C_n}$ be a mixed cycle with vertex set $\{u_1,u_2,\ldots,u_n\}$. The \textit{value} of $\widetilde{C_n}$, denoted by $h(\widetilde{C_n})$, is defined as $h_{12}h_{23}\ldots h_{n1}$. Clearly, if in one direction the value of $\widetilde{C_n}$ is $h$, then its value is $\overline{h}$ (the conjugate number of $h$) for the reverse direction. Thus, we refer to the mixed cycle $\widetilde{C_n}$ with value 1 (resp. $-1$) as a positive (resp. negative) mixed cycle regardless of the direction of $\widetilde{C_n}$. For some orientation of $\widetilde{C_n}$, the \textit{signature} of $\widetilde{C_n}$, denoted by $\sigma(\widetilde{C_n})$, is defined as the difference between the number of its forward and backward directed edges in $\widetilde{C_n}$. Hence, a mixed cycle is positive (resp. negative) if and only if its signature with respect to an arbitrary direction is congruence to 0 (resp. 2) modulo 4. A mixed graph $\widetilde{G}$ is said to be \textit{positive} if all its mixed cycles are positive, otherwise it is called \textit{non-positive}.

\begin{thm}\label{thm1.2}
Let $\widetilde{G}$ be a connected mixed graph with a cut vertex $v$ and without pendant vertices. Then $p(\widetilde{G})=2$ if and only if $\widetilde{G}$ is switching equivalent to one of the following mixed graphs:
\begin{wst}
  \item[{\rm (i)}] $\widetilde{K}_{m_1, m_2,\ldots, m_{x}}\bullet v\bullet\widetilde{K}_{m_1', m_2',\ldots, m'_y}$, where $\widetilde{K}_{m_1, m_2,\ldots, m_{x}}$ (resp. $\widetilde{K}_{m_1', m_2',\ldots, m'_y}$) is not a mixed star, and $\widetilde{K}_{m_1, m_2,\ldots, m_{x}}$ (resp. $\widetilde{K}_{m_1', m_2',\ldots, m'_y}$) is positive or switching equivalent to some $\overrightarrow{C_3}(t_1, t_2, t_3)$.
  \item[{\rm (ii)}] $K(q_1,\ldots, q_r; n_1, \ldots, n_k; p)$, where $r\geq2, k\geq2, p\geq1$, satisfying one of the following conditions:
\begin{wst}
  \item[{\rm (1)}] $p=1$ and either $n_1=1$ and $k-p\geq2$ or $n_1\geq2$;
  \item[{\rm (2)}] $p\geq2$ and $k-p\leq1$;
  \item[{\rm (3)}] $p\geq2, k-p\geq2$ and $\frac{1}{r}+\frac{1}{p}+\frac{1}{k-p-1}\geq1$.
\end{wst}
  \item[{\rm (iii)}] $K(q_1,\ldots, q_r; n_1, \ldots, n_k; a^i, b^{-i}, 0^{1}, 0^{-1}, s)$, where $r\geq2, k\geq2, a\geq b\geq1, s=k-a-b$, satisfying $a=b=1$ and $r=2$.
  \item[{\rm (iv)}] $K(q_1,\ldots, q_r; n_1, \ldots, n_k; a^i, 0^{-i}, c^{1}, 0^{-1}, s)$, where $r\geq2, k\geq2, a\geq c\geq1, s=k-a-c$, satisfying one of the following conditions:
\begin{wst}
  \item[{\rm (1)}] $a=c=1$ and either $s=0$ or $s=1$;
  \item[{\rm (2)}] $a=c=1, s=2$ and either $r=3$ or $r=4$;
  \item[{\rm (3)}] $a=c=1$, $s=3$ and $r=3$;
  \item[{\rm (4)}] $a=c=1$, $s\geq2$ and $r=2$;
  \item[{\rm (5)}] $a=c=2, s=0$ and $2\leq r\leq4$;
  \item[{\rm (6)}] $a=c=2, s=1$ and $r=2$;
  \item[{\rm (7)}] $a=3, s=0$ and $r=c=2$;
  \item[{\rm (8)}] $a=4, c=r=2$ and $s=0$;
  \item[{\rm (9)}] $a\geq2, c=1$ and $\frac{as-1}{a+s}\leq\frac{1}{r-1}$.
\end{wst}
\end{wst}
\end{thm}
\begin{figure}[h!]
\begin{center}
\psfrag{4}{$\widetilde{G}$}\psfrag{5}{$v$}
\psfrag{6}{$\cdots$}\psfrag{7}{$\widetilde{G'}$}\psfrag{a}{Figure 3: Two mixed graphs $\widetilde{G}$ and $\widetilde{G'}$ with positive inertia index 2.}
\psfrag{b}{$n_1$}\psfrag{c}{$n_2$}\psfrag{d}{$n_3$}\psfrag{e}{$n_4(=1)$}\psfrag{8}{$\vdots$}\psfrag{1}{$q_1$}\psfrag{2}{$q_3$}\psfrag{3}{$q_2$}
\includegraphics[width=160mm]{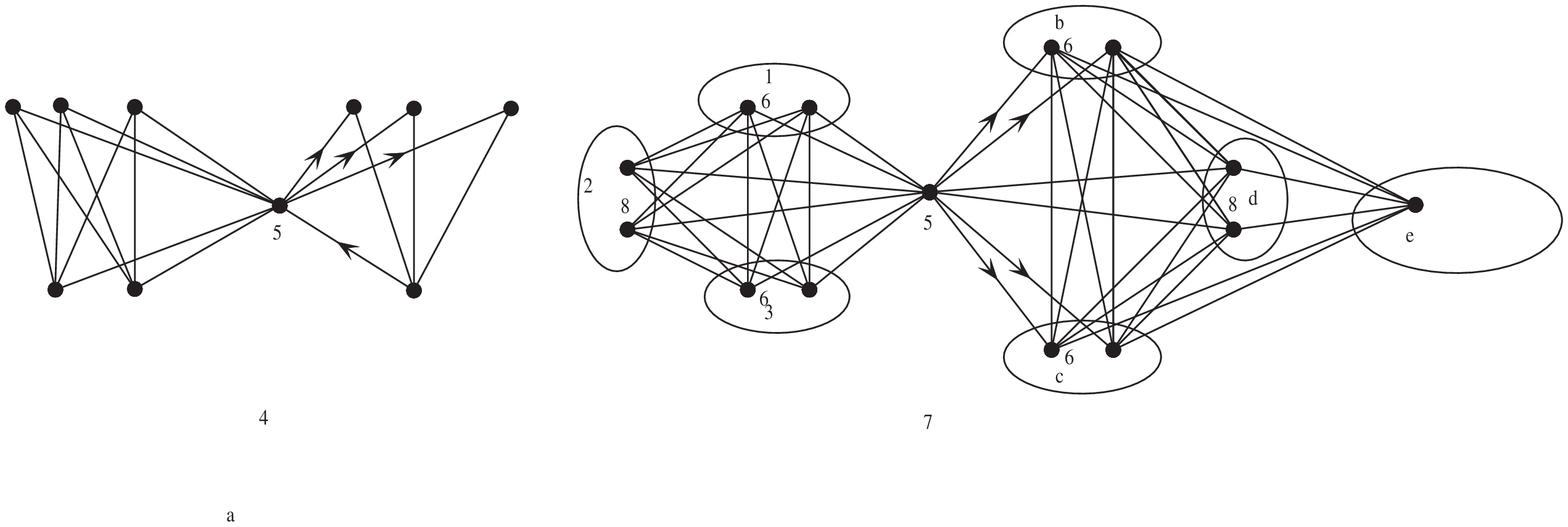} \\
\end{center}
\end{figure}
\begin{ex}
$\widetilde{G}=K(3, 2; 3, 1; 1^i, 1^{-i}, 0^1, 0^{-1}, 0)$ satisfies Theorem \ref{thm1.2} (iii) and $$\widetilde{G'}=K(q_1, q_2, q_3; n_1, n_2, n_3, n_4; 2^i, 0^{-i}, 1^1, 0^{-1}, 1)$$ satisfies Theorem \ref{thm1.2} (iv) (9) (see Figure. 3). It is easy to see that $p(\widetilde{G})=p(\widetilde{G'})=2$.
\end{ex}

{In the rest of this section we recall some known results. In Section 2 we give the
proof of Theorem 1.1. In Section 3 we first establish some technical lemmas that help us characterize the
mixed graphs with cut vertices having exactly two positive eigenvalues and then present the proof of Theorem 1.2.}
\subsection{\normalsize Preliminaries}
In this {subsection}, we introduce some helpful lemmas and known results.
\begin{lem}\cite{horn}\label{lem1.4}
(Sylvester's law of inertia) If two Hermitian matrices $M$ and $N$ are congruent, then they have the same positive (resp. negative) inertia indices.
\end{lem}
\begin{lem}\cite{lxl}\label{lem1.5}
Let $\widetilde{G}$ be a mixed graph. Then the following conditions are equivalent.
\begin{wst}
  \item[{\rm (1)}] Every induced mixed cycle of $\widetilde{G}$ is positive.
  \item[{\rm (2)}] $\widetilde{G}$ is positive.
  \item[{\rm (3)}] $\widetilde{G}$ is switching equivalent to its underlying graph.
\end{wst}
\end{lem}
\begin{lem}\cite{WW,yuan}\label{lemp}
Let $v_1$ be a pendant vertex of mixed graph $\widetilde{G}$ with neighbor $v_2$. Then
$$p(\widetilde{G})=p(\widetilde{G}-v_1-v_2)+1, \ \ \ n(\widetilde{G})=n(\widetilde{G}-v_1-v_2)+1, \ \ \ \eta(\widetilde{G})=\eta(\widetilde{G}-v_1-v_2).$$
\end{lem}
\begin{lem}\cite{WW}\label{leml}
Let $\widetilde{G}$ be a mixed graph {with $u\in V(G)$}. Then
\begin{align*}
p(\widetilde{G})-1\leq p(\widetilde{G}-u)\leq p(\widetilde{G}), \ \ \ n(\widetilde{G})-1\leq n(\widetilde{G}-u)\leq n(\widetilde{G}).
\end{align*}
\end{lem}
\begin{lem}\cite{yuan}\label{lem1.6}
Let $v$ be a cut vertex of a mixed graph $\widetilde{G}$, and let $\widetilde{G_1}, \widetilde{G_2}, \ldots, \widetilde{G_l}$ be all the components of $\widetilde{G}-v$.
\begin{wst}
  \item[{\rm (1)}] If there exists a component $\widetilde{G_i}$ satisfying $rk(\widetilde{G_i}+v)=rk(\widetilde{G_i})+2$, then $${\rm In}(\widetilde{G})={\rm In}(\widetilde{G}-v)+(1, 1, -1)=\sum_{j=1}^l{\rm In}(\widetilde{G_j})+(1, 1, -1).$$
  \item[{\rm (2)}] If there exists a component $\widetilde{G_i}$ satisfying $rk(\widetilde{G_i}+v)=rk(\widetilde{G_i})$, then $${\rm In}(\widetilde{G})={\rm In}(\widetilde{G}_i)+{\rm In}(\widetilde{G}-\widetilde{G}_i).$$
      In particular, if $rk(\widetilde{G_i}+v)=rk(\widetilde{G_i})$ for each $i=1, 2,\ldots, l$, then
      $${\rm In}(\widetilde{G})={\rm In}(\widetilde{G}-v)+(0, 0, 1)=\sum_{j=1}^l{\rm In}(\widetilde{G_j})+(0, 0, 1).$$
\end{wst}
\end{lem}
\begin{lem}\cite{34}\label{si}
Let $\widetilde{C_n}$ be a mixed cycle of order $n$. Then
\begin{center}
$\eta(\widetilde{C_n})=\left\{
  \begin{array}{ll}
    1, & \hbox{if $n$ is odd, $\sigma(\widetilde{C_n})$ is odd;} \\[1mm]
    0, & \hbox{if $n$ is odd, $\sigma(\widetilde{C_n})$ is even;} \\ [1mm]
    0, & \hbox{if $n$ is even, $\sigma(\widetilde{C_n})$ is odd;} \\ [1mm]
    0, & \hbox{if $n$ is even, $\sigma(\widetilde{C_n})$ is even, $n+\sigma(\widetilde{C_n})\equiv 2 (mod 4)$;} \\ [1mm]
    2, & \hbox{if $n$ is even, $\sigma(\widetilde{C_n})$ is even, $n+\sigma(\widetilde{C_n})\equiv 0 (mod 4)${.}}
  \end{array}
\right.$
\end{center}
\end{lem}
\begin{lem}\cite{yuan}\label{lem1.9}
Let $\widetilde{G}$ be a mixed graph. Then $p(\widetilde{G})=1$ if and only if the subgraph induced by all non-isolated vertices of $\widetilde{G}$ is
switching equivalent either to a complete multipartite graph or to some $\overrightarrow{C_3}(t_1, t_2, t_3)$.
\end{lem}

\section{\normalsize Proof for Theorem \ref{thm1.1}}
\noindent{\bf Proof of Theorem \ref{thm1.1}:}
The sufficiency can be verified by Lemma \ref{lemp}, and we only prove necessity.

\textit{``Necessity''.} If $\widetilde{G}$ is a mixed graph with pendant vertices, without loss of generality, let $v_1$ be a pendant vertex of $\widetilde{G}$ with neighbor $v_2$. Then by Lemma \ref{lemp}, we have $p(\widetilde{G})=p(\widetilde{G}-v_1-v_2)+1$. The condition $p(\widetilde{G})=2$ implies $p(\widetilde{G}-v_1-v_2)=1$. Thus, by Lemma \ref{lem1.9}, the subgraph, say $\widetilde{F}$, induced by all non-isolated vertices of $\widetilde{G}-v_1-v_2$ is switching equivalent either to a complete multipartite graph or to some $\overrightarrow{C_3}(t_1, t_2, t_3)$. Hence, $p(\widetilde{F})=1$. It is easy to see that $\widetilde{G}$ is obtained by adding some undirected edges (possibly empty) and directed edges (possibly empty) with arbitrary direction between the center of a mixed star and some vertices of $\widetilde{F}$.
{This completes the proof.}
\qed
\section{\normalsize Proof for Theorem \ref{thm1.2}}
\subsection{\normalsize Technical lemmas}
In this {subsection} we present a few technical lemmas aiming to provide some fundamental characterizations of mixed graphs with cut vertices and without pendant vertices having exactly two positive eigenvalues.

Two vertices $u, w\in V(\widetilde{G})$ are called \textit{twins} if $\widetilde{G}$ is switching equivalent to $\widetilde{G'}$ such that the rows of $H(\widetilde{G'})$ indexed by $u$ and $w$ are same. The relation of being twins is an equivalence relation on vertex set of $\widetilde{G}$. Denote by $[u]$ the equivalence class containing $u$. The \textit{twin reduction graph} of $\widetilde{G}$, written as $T_{\widetilde{G}}$, is a mixed graph whose vertices are the equivalence classes and $[u][w]\in E_0(T_{\widetilde{G}})$ if and only if $uw\in E_0(\widetilde{G})$ and $\overrightarrow{[u][w]}\in E_1(T_{\widetilde{G}})$ if and only if $\overrightarrow{uw}\in E_1(\widetilde{G})$.
\begin{lem}\cite{mo}\label{twin}
Let $\widetilde{G}$ and $\widetilde{G'}$ be mixed graphs with the same underlying graph. Then they are switching equivalent if and only if $T_{\widetilde{G}}$ and $T_{\widetilde{G'}}$ are switching equivalent.
\end{lem}
\begin{lem}\cite{34}\label{add}
Adding or removing twins preserves the $H$-rank of a mixed graph, and also preserves the positive and negative inertia index of a mixed graph.
\end{lem}
\begin{lem}\label{rank}
$rk(\overrightarrow{C_3}(t_1, t_2, t_3))=2$.
\end{lem}
\begin{proof}
It is easy to see that $T_{\overrightarrow{C_3}(t_1, t_2, t_3)}=\overrightarrow{C_3}(1, 1, 1)$ and $rk(\overrightarrow{C_3}(1, 1, 1))=2$. By Lemma \ref{add}, we have
$$rk(\overrightarrow{C_3}(t_1, t_2, t_3))=rk(\overrightarrow{C_3}(1, 1, 1))=2.$$
\end{proof}
\begin{lem}\cite{34}\label{ev}
Let $\widetilde{G}$ be a connected mixed graph. Then $rk(\widetilde{G})=3$ if and only if $T_{\widetilde{G}}$ is an even triangle.
\end{lem}
\begin{lem}
Let $\widetilde{G}$ be a positive complete $k$-partite mixed graph with one partition class containing exactly one vertex. If $k\geq3$ and $v\in V(\widetilde{G})$ is in one partition class containing exactly one vertex, then $rk(\widetilde{G})=rk(\widetilde{G}-v)+1$.
\end{lem}
\begin{proof}
By Lemma \ref{lem1.5}, we have $\widetilde{G}$ is switching equivalent to its underlying graph. Then the vertices in the same partition class of $\widetilde{G}$ are twins. Thus, by Lemma \ref{twin}, we have $T_{\widetilde{G}}$ is switching equivalent to $K_k$. Using Lemma \ref{add}, we have $rk(\widetilde{G})=rk(T_{\widetilde{G}})=rk(K_k)=k$. Similarly, we have $rk(\widetilde{G}-v)=rk(K_{k-1})=k-1$. Hence, $rk(\widetilde{G})=rk(\widetilde{G}-v)+1$.
\end{proof}
\begin{lem}\label{lem3.3}
Let $\widetilde{M_1}$ be a mixed graph with a vertex $v'$ and $\widetilde{M_2}$ be another mixed graph with a vertex $v''$, and let $\widetilde{G}=\widetilde{M_1}\bullet v\bullet\widetilde{M_2}$. Then $p(\widetilde{M_1}-v')+p(\widetilde{M_2}-v'')\leq p(\widetilde{G})\leq p(\widetilde{M_1})+p(\widetilde{M_2})$.
\end{lem}
\begin{proof}
By Lemma \ref{leml}, we obtain
\begin{align*}
p(\widetilde{M_1}-v')\leq p(\widetilde{M_1})\leq p(\widetilde{M_1}-v')+1.
\end{align*}
Similarly,
\begin{align*}
p(\widetilde{M_2}-v'')\leq p(\widetilde{M_2})\leq p(\widetilde{M_2}-v'')+1.
\end{align*}
Lemma \ref{leml} applying to the induced subgraph $\widetilde{G}-v=(\widetilde{M_1}-v')\bigcup(\widetilde{M_2}-v'')$ of $\widetilde{G}$ yields that
\begin{align*}
p(\widetilde{M_1}-v')+p(\widetilde{M_2}-v'')\leq p(\widetilde{G})\leq p(\widetilde{M_1}-v')+p(\widetilde{M_2}-v'')+1.
\end{align*}
If $p(\widetilde{M_1})=p(\widetilde{M_1}-v')+1$, then
\begin{align*}
p(\widetilde{G})\leq p(\widetilde{M_1}-v')+p(\widetilde{M_2}-v'')+1\leq p(\widetilde{M_1})+p(\widetilde{M_2}).
\end{align*}
If $p(\widetilde{M_2})=p(\widetilde{M_2}-v'')+1$, we also have
\begin{align*}
p(\widetilde{G})\leq p(\widetilde{M_1}-v')+p(\widetilde{M_2}-v'')+1\leq p(\widetilde{M_1})+p(\widetilde{M_2}).
\end{align*}
Suppose now $p(\widetilde{M_1})=p(\widetilde{M_1}-v')$ and $p(\widetilde{M_2})=p(\widetilde{M_2}-v'')$ and consider the Hermitian adjacency matrix $H(\widetilde{G})$ of $\widetilde{G}$:
\begin{align*}
H(\widetilde{G})={\left(
        \begin{array}{ccc}
        H(\widetilde{M_1}-v') & \alpha & 0 \\[3pt]
        \alpha^{*} & 0 & \beta^{*} \\[3pt]
        0 & \beta & H(\widetilde{M_2}-v'')
        \end{array}
        \right)}.
\end{align*}
The condition $p(\widetilde{M_1})=p(\widetilde{M_1}-v')$ implies that $rk(\widetilde{M_1})\leq rk(\widetilde{M_1}-v')+1$, thus the equation $H(\widetilde{M_1}-v')X=\alpha$ has a solution, say $X_0$. Similarly, suppose $Y_0$ is a solution of the equation $H(\widetilde{M_2}-v'')Y=\beta$ and let
\begin{align*}\label{eq1.2}
S={\left(
        \begin{array}{ccc}
        I & -X_0 & 0 \\[3pt]
        0 & 1 & 0 \\[3pt]
        0 & -Y_0 & I
        \end{array}
        \right)}.
\end{align*}
Then
\begin{align*}
S^*H(\widetilde{G})S={\left(
        \begin{array}{ccc}
        H(\widetilde{M_1}-v') & 0 & 0 \\[3pt]
        0 & -\alpha^*X_0-\beta^*Y_0 & 0 \\[3pt]
        0 & 0 & H(\widetilde{M_2}-v'')
        \end{array}
        \right)}.
\end{align*}
From $p(\widetilde{M_1})=p(\widetilde{M_1}-v')$ and $p(\widetilde{M_2})=p(\widetilde{M_2}-v'')$ it follows that $-\alpha^*X_0\leq0$ and $-\beta^*Y_0\leq0$. Thus by Lemma \ref{lem1.4},
$p(\widetilde{G})=p(S^*H(\widetilde{G})S)=p(\widetilde{M_1}-v')+p(\widetilde{M_2}-v'')=p(\widetilde{M_1})+p(\widetilde{M_2})$.
This completes the proof.
\end{proof}
\begin{lem}\label{lem3.2}
Let $\widetilde{K}_{m_1, m_2,\ldots, m_x}$ and $\widetilde{K}_{m'_1, m'_2,\ldots, m'_y}$ be mixed complete multipartite graphs without pendant vertices. Let $v'$ (resp. $v''$) be a vertex of $\widetilde{K}_{m_1, m_2,\ldots, m_x}$ (resp. $\widetilde{K}_{m'_1, m'_2,\ldots, m'_y}$). If $p(\widetilde{K}_{m_1, m_2,\ldots, m_x})=1$ and $p(\widetilde{K}_{m'_1, m'_2,\ldots, m'_y})=1$, then $p(\widetilde{K}_{m_1, m_2,\ldots, m_x}\bullet v\bullet\widetilde{K}_{m'_1, m'_2,\ldots, m'_y})=2$.
\end{lem}
\begin{proof}
Since $\widetilde{K}_{m_1, m_2,\ldots, m_x}$ and $\widetilde{K}_{m'_1, m'_2,\ldots, m'_y}$ are mixed complete multipartite graphs without pendant vertices, we have $K_{m_1, m_2,\ldots, m_x}$ and $K_{m'_1, m'_2,\ldots, m'_y}$ are not mixed stars. Thus, $\widetilde{K}_{m_1, m_2,\ldots, m_x}-v'$ (resp. $\widetilde{K}_{m'_1, m'_2,\ldots, m'_y}-v''$) has at least one edge. It follows that $p(\widetilde{K}_{m_1, m_2,\ldots, m_x}-v')\geq1$ and $p(\widetilde{K}_{m'_1, m'_2,\ldots, m'_y}-v'')\geq1$. By Lemma \ref{leml}, we have $p(\widetilde{K}_{m_1, m_2,\ldots, m_x}-v')=1$ and $p(\widetilde{K}_{m'_1, m'_2,\ldots, m'_y}-v'')=1$. In view of Lemma \ref{lem3.3}, we have
\begin{align*}
2=&p(\widetilde{K}_{m_1, m_2,\ldots, m_x}-v')+p(\widetilde{K}_{m'_1, m'_2,\ldots, m'_y}-v'')\\
\leq&p(\widetilde{K}_{m_1, m_2,\ldots, m_x}\bullet v\bullet\widetilde{K}_{m'_1, m'_2,\ldots, m'_y})\\
\leq&p(\widetilde{K}_{m_1, m_2,\ldots, m_x})+p(\widetilde{K}_{m'_1, m'_2,\ldots, m'_y})\\
=&2.
\end{align*}
This completes the proof.
\end{proof}
Let ${\bf J}_{n\times m}$ and ${\bf 0}_{n\times m}$ be respectively the all-one and the all-zero $n\times m$ matrices. Let ${\bf 1}_{n}={\bf J}_{n\times 1}$. We often simply write ${\bf J}, {\bf 0}$ and ${\bf 1}$, respectively, if the order of these matrices is clear from the context.
\begin{lem}\label{3.8}
For $r\geq2, k\geq2, a\geq b, a\geq1$, and $s=k-a-b$, the positive inertia index of $\widetilde{G}=K(r; k; a^i, b^{-i}, 0^{1}, 0^{-1}, s)$ is 2 if and only if one of the following conditions holds:
\begin{wst}
  \item[{\rm (1)}] $a=1, b=0$;
  \item[{\rm (2)}] $a=b=1$ and $r=2$;
  \item[{\rm (3)}] $a\geq2, b=0$ and $s\leq1$;
  \item[{\rm (4)}] $a\geq2, b=0, s\geq2$ and $\frac{1}{r}+\frac{1}{a}+\frac{1}{s-1}\geq1$.
\end{wst}
\end{lem}
\begin{proof}
The Hermitian adjacency matrix of $\widetilde{G}=K(r; k; a^i, b^{-i}, 0^{1}, 0^{-1}, s)$ can be written as
\begin{align*}
H(\widetilde{G})={\left(
        \begin{array}{ccccc}
        H(K_r) & \mathbf{1} & {\bf 0} & {\bf 0} & {\bf 0} \\[3pt]
        \mathbf{1}^{\top} & 0 & i\mathbf{1}^{\top} & -i\mathbf{1}^{\top} & {\bf 0} \\[3pt]
        {\bf 0} & -i\mathbf{1} & H(K_a) & {\bf J} & {\bf J} \\[3pt]
        {\bf 0} & i\mathbf{1} & {\bf J} & H(K_b) & {\bf J} \\[3pt]
        {\bf 0} & {\bf 0} & {\bf J} & {\bf J} & H(K_s)
        \end{array}
        \right)}.
\end{align*}

{\bf{Case 1.}}\ $a=1$ and $b=0$. If $s=1$, then $\widetilde{G}$ has a pendant vertex, by Theorem \ref{lemp}, it is easy to see that $p(\widetilde{G})=2$. If $s\geq2$, then
\begin{align*}
H(\widetilde{G})={\left(
        \begin{array}{cccc}
        H(K_r) & \mathbf{1} & {\bf 0} & {\bf 0} \\[3pt]
        \mathbf{1}^{\top} & 0 & i & {\bf 0} \\[3pt]
        {\bf 0} & -i & 0 & \mathbf{1}^{\top} \\[3pt]
        {\bf 0} & {\bf 0} & \mathbf{1} & H(K_s)
        \end{array}
        \right)}.
\end{align*}
Let
\begin{align*}
C={\left(
        \begin{array}{cccc}
        I_r & -\frac{1}{r-1}\mathbf{1} & {\bf 0} & {\bf 0} \\[3pt]
        {\bf 0} & 1 & 0 & {\bf 0} \\[3pt]
        {\bf 0} & 0 & 1 & {\bf 0} \\[3pt]
        {\bf 0} & {\bf 0} & -\frac{1}{s-1}\mathbf{1} & I_s
        \end{array}
        \right)}.
\end{align*}
Then
\begin{align*}
C^*H(\widetilde{G})C={\left(
        \begin{array}{cccc}
        H(K_r) & {\bf 0} & {\bf 0} & {\bf 0} \\[3pt]
        {\bf 0} & -\frac{r}{r-1} & i & {\bf 0} \\[3pt]
        {\bf 0} & -i & -\frac{s}{s-1} & {\bf 0} \\[3pt]
        {\bf 0} & {\bf 0} & {\bf 0} & H(K_s)
        \end{array}
        \right)}.
\end{align*}
It is easy to see that the eigenvalues of
\begin{align*}
{\left(
        \begin{array}{cc}
        -\frac{r}{r-1} & i \\[3pt]
        -i & -\frac{s}{s-1}
        \end{array}
        \right)}
\end{align*}
are both negative, therefore
$$ p(C^*H(\widetilde{G})C)=p(H(K_r))+p(H(K_s))=2.$$
Thus in this case, the positive inertia index of $K(r; k; a^i, b^{-i}, 0^{1}, 0^{-1}, s)$ must be 2.

{\bf{Case 2.}}\ $a=b=1$. If $s=0$, then
\begin{align*}
H(\widetilde{G})={\left(
        \begin{array}{cccc}
        H(K_r) & \mathbf{1} & {\bf 0} & {\bf 0} \\[3pt]
        \mathbf{1}^{\top} & 0 & i & -i \\[3pt]
        {\bf 0} & -i & 0 & 1 \\[3pt]
        {\bf 0} & i & 1 & 0
        \end{array}
        \right)}.
\end{align*}
Let
\begin{align*}
C_1={\left(
        \begin{array}{cccc}
        I_r & -\frac{1}{r-1}\mathbf{1} & {\bf 0} & {\bf 0} \\[3pt]
        {\bf 0} & 1 & 0 & 0 \\[3pt]
        {\bf 0} & -i & 1 & 0 \\[3pt]
        {\bf 0} & i & 0 & 1
        \end{array}
        \right)}.
\end{align*}
Then
\begin{align*}
C_1^*H(\widetilde{G})C_1={\left(
        \begin{array}{cccc}
        H(K_r) & {\bf 0} & {\bf 0} & {\bf 0} \\[3pt]
        {\bf 0} & 2-\frac{r}{r-1} & 0 & 0 \\[3pt]
        {\bf 0} & 0 & 0 & 1 \\[3pt]
        {\bf 0} & 0 & 1 & 0
        \end{array}
        \right)}.
\end{align*}
Thus $p(\widetilde{G})=2$ if and only if $r=2$.

If $s=1$, then
\begin{align*}
H(\widetilde{G})={\left(
        \begin{array}{ccccc}
        H(K_r) & \mathbf{1} & {\bf 0} & {\bf 0} & {\bf 0} \\[3pt]
        \mathbf{1}^{\top} & 0 & i & -i & 0 \\[3pt]
        {\bf 0} & -i & 0 & 1 & 1 \\[3pt]
        {\bf 0} & i & 1 & 0 & 1 \\[3pt]
        {\bf 0} & 0 & 1 & 1 & 0
        \end{array}
        \right)}.
\end{align*}
Let
\begin{align*}
C_2={\left(
        \begin{array}{ccccc}
        I_r & -\frac{1}{r-1}\mathbf{1} & 0 & 0 & 0 \\[3pt]
        0 & 1 & 0 & 0 & 0 \\[3pt]
        0 & 0 & 1 & 0 & 0 \\[3pt]
        0 & 0 & -1 & 1 & 0 \\[3pt]
        0 & -i & -1 & 0 & 1
        \end{array}
        \right)}.
\end{align*}
Then
\begin{align*}
C_2^*H(\widetilde{G})C_2={\left(
        \begin{array}{ccccc}
        H(K_r) & {\bf 0} & {\bf 0} & {\bf 0} & {\bf 0} \\[3pt]
        {\bf 0} & -\frac{r}{r-1} & 2i & 0 & 0 \\[3pt]
        {\bf 0} & -2i & -2 & 0 & 0 \\[3pt]
        {\bf 0} & 0 & 0 & 0 & 1 \\[3pt]
        {\bf 0} & 0 & 0 & 1 & 0
        \end{array}
        \right)}.
\end{align*}
Thus $p(\widetilde{G})=2$ if and only if
\begin{align*}
{\left(
        \begin{array}{cc}
        -\frac{r}{r-1} & 2i \\[3pt]
        -2i & -2
        \end{array}
        \right)}.
\end{align*}
has no positive eigenvalue, if and only if $r=2$.

If $s\geq2$, let
\begin{align*}
C_3={\left(
        \begin{array}{ccccc}
        I_r & -\frac{1}{r-1}\mathbf{1} & {\bf 0} & {\bf 0} & {\bf 0} \\[3pt]
        {\bf 0} & 1 & 0 & 0 & {\bf 0} \\[3pt]
        {\bf 0} & 0 & 1 & 0 & {\bf 0} \\[3pt]
        {\bf 0} & 0 & 0 & 1 & {\bf 0} \\[3pt]
        {\bf 0} & {\bf 0} & -\frac{1}{s-1}\mathbf{1} & -\frac{1}{s-1}\mathbf{1} & I_s
        \end{array}
        \right)}.
\end{align*}
Then
\begin{align*}
C_3^*H(\widetilde{G})C_3={\left(
        \begin{array}{ccccc}
        H(K_r) & {\bf 0} & {\bf 0} & {\bf 0} & {\bf 0} \\[3pt]
        {\bf 0} & -\frac{r}{r-1} & i & -i & {\bf 0} \\[3pt]
        {\bf 0} & -i & -\frac{s}{s-1} & -\frac{1}{s-1} & {\bf 0} \\[3pt]
        {\bf 0} & i & -\frac{1}{s-1} & -\frac{s}{s-1} & {\bf 0} \\[3pt]
        {\bf 0} & {\bf 0} & {\bf 0} & {\bf 0} & H(K_s)
        \end{array}
        \right)}.
\end{align*}
If $r>2$, since the determinant of
\begin{align*}
{\left(
        \begin{array}{ccc}
        -\frac{r}{r-1} & i & -i \\[3pt]
        -i & -\frac{s}{s-1} & -\frac{1}{s-1} \\[3pt]
        i & -\frac{1}{s-1} & -\frac{s}{s-1}
        \end{array}
        \right)}
\end{align*}
is positive, the positive inertia index of this matrix is at least 1, and thus $p(\widetilde{G})\geq3$. If $r=2$, it is easy to see that the above matrix has 0 as an eigenvalue and the other two eigenvalues are negative, and thus $p(\widetilde{G})=2$. Consequently, when $a=b=1$, $p(\widetilde{G})=2$ if and only if
$r=2$.

{\bf{Case 3.}}\ $a\geq2$. If $k=a=2$, then by Lemma \ref{lem3.3},
$$2=p(K_r)+p(K_k)\leq p(\widetilde{G})\leq p(K_{r+1})+p(\widetilde{K_{k+1}})=2,$$
and thus $p(\widetilde{G})=2$. Suppose $k\geq3$. Let $U_n$ be the strictly upper triangular matrix of order $n$ whose $(i, j)$-entries are 1 for all $1\leq i<j\leq n$. Let $D=diag(\frac{-1}{a-1}, \frac{-1}{a}, \cdots, \frac{-1}{k-2})$. Set
\begin{align*}
C_4={\left(
        \begin{array}{ccc}
        I_r & {\bf 0} & {\bf 0} \\[3pt]
        {\bf 0} & 1 & {\bf 0} \\[3pt]
        {\bf 0} & {\bf 0} & P
        \end{array}
        \right)},
\end{align*}
where $P$ is the following matrix
\begin{align*}
P={\left(
        \begin{array}{cccc}
        I_a & {\bf J}D \\[3pt]
        {\bf 0} & I_{k-a}+U_{k-a}D
        \end{array}
        \right)}.
\end{align*}
Then $C_4^*H(\widetilde{G})C_4$ is the following matrix:
\begin{align*}
{\left(
        \begin{array}{cccccccccc}
        H(K_r) & \mathbf{1} & {\bf 0} & {\bf 0} & {\bf 0} & \cdots & {\bf 0} & {\bf 0} & \cdots & {\bf 0} \\[3pt]
        \mathbf{1}^{\top} & 0 & i\mathbf{1}^{\top} & \frac{(1-2a)i}{a-1} & \frac{(1-2a)i}{a} & \cdots & \frac{(1-2a)i}{a+b-2} & \frac{(b-a)i}{a+b-1} & \cdots & \frac{(b-a)i}{k-2} \\[3pt]
        {\bf 0} & -i\mathbf{1} & H(K_a) & 0 & 0 & \cdots & 0 & 0 & \cdots & 0 \\[3pt]
        {\bf 0} & \frac{(2a-1)i}{a-1} & 0 & \frac{-a}{a-1} & 0 & \cdots & 0 & 0 & \cdots & 0 \\[3pt]
        {\bf 0} & \frac{(2a-1)i}{a} & 0 & 0 & \frac{-(a+1)}{a} & \cdots & 0 & 0 & \cdots & 0 \\[3pt]
        \vdots & \vdots & \vdots & \vdots & \vdots & \ddots & \vdots & \vdots & \cdots & \vdots \\[3pt]
        {\bf 0} & \frac{(2a-1)i}{a+b-2} & 0 & 0 & 0 & \cdots & \frac{-(a+b-1)}{a+b-2} & 0 & \cdots & 0 \\[3pt]
        {\bf 0} & \frac{(a-b)i}{a+b-1} & 0 & 0 & 0 & \cdots & 0 & \frac{-(a+b)}{a+b-1} & \cdots & 0 \\[3pt]
        \vdots & \vdots & \vdots & \vdots & \vdots & \ddots & \vdots & \vdots & \cdots & \vdots \\[3pt]
        {\bf 0} & \frac{(a-b)i}{k-2} & 0 & 0 & 0 & \cdots & 0 & 0 & \cdots & \frac{-(k-1)}{k-2}
        \end{array}
        \right)}.
\end{align*}
Let
\begin{align*}
C_5={\left(
        \begin{array}{cccccccccc}
        I_r & -\frac{1}{r-1}\mathbf{1} & {\bf 0} & {\bf 0} & {\bf 0} & \ldots & {\bf 0} & {\bf 0} & \ldots & {\bf 0} \\[3pt]
        {\bf 0} & 1 & 0 & 0 & 0 & \ldots & 0 & 0 & \ldots & 0 \\[3pt]
        {\bf 0} & \frac{i}{a-1}\mathbf{1} & I_a & 0 & 0 & \ldots & 0 & 0 & \ldots & 0 \\[3pt]
        {\bf 0} & \frac{(2a-1)i}{a} & 0 & 1 & 0 & \ldots & 0 & 0 & \ldots & 0 \\[3pt]
        {\bf 0} & \frac{(2a-1)i}{a+1} & 0 & 0 & 1 & \ldots & 0 & 0 & \ldots & 0 \\[3pt]
        \vdots & \vdots & \vdots & \vdots & \vdots & \ddots & \vdots & \vdots & \ldots & \vdots \\[3pt]
        {\bf 0} & \frac{(2a-1)i}{a+b-1} & 0 & 0 & 0 & \ldots & 1 & 0 & \ldots & 0 \\[3pt]
        {\bf 0} & \frac{(a-b)i}{a+b} & 0 & 0 & 0 & \ldots & 0 & 1 & \ldots & 0 \\[3pt]
        \vdots & \vdots & \vdots & \vdots & \vdots & \ddots & \vdots & \vdots & \ldots & \vdots \\[3pt]
        {\bf 0} & \frac{(a-b)i}{k-1} & 0 & 0 & 0 & \ldots & 0 & 0 & \ldots & 1
        \end{array}
        \right)}.
\end{align*}
Then $C_5^*C_4^*H(\widetilde{G})C_4C_5$ is the following matrix:
\begin{align*}
{\left(
        \begin{array}{cccccccccc}
        H(K_r) & {\bf 0} & {\bf 0} & {\bf 0} & {\bf 0} & \ldots & {\bf 0} & {\bf 0} & \ldots & {\bf 0} \\[3pt]
        {\bf 0} & \rho & 0 & 0 & 0 & \ldots & 0 & 0 & \ldots & 0 \\[3pt]
        {\bf 0} & 0 & H(K_a) & 0 & 0 & \ldots & 0 & 0 & \ldots & 0 \\[3pt]
        {\bf 0} & 0 & 0 & \frac{-a}{a-1} & 0 & \ldots & 0 & 0 & \ldots & 0 \\[3pt]
        {\bf 0} & 0 & 0 & 0 & \frac{-(a+1)}{a} & \ldots & 0 & 0 & \ldots & 0 \\[3pt]
        \vdots & \vdots & \vdots & \vdots & \vdots & \ddots & \vdots & \vdots & \ldots & \vdots \\[3pt]
        {\bf 0} & 0 & 0 & 0 & 0 & \ldots & \frac{-(a+b-1)}{a+b-2} & 0 & \ldots & 0 \\[3pt]
        {\bf 0} & 0 & 0 & 0 & 0 & \ldots & 0 & \frac{-(a+b)}{a+b-1} & \ldots & 0 \\[3pt]
        \vdots & \vdots & \vdots & \vdots & \vdots & \ddots & \vdots & \vdots & \ldots & \vdots \\[3pt]
        {\bf 0} & 0 & 0 & 0 & 0 & \ldots & 0 & 0 & \ldots & \frac{-(k-1)}{k-2}
        \end{array}
        \right)},
\end{align*}
where
\begin{align*}
\rho&=(2a-1)^2\left(\frac{1}{a-1}-\frac{1}{a+b-1}\right)+(a-b)^2\left(\frac{1}{a+b-1}-\frac{1}{k-1}\right)-\frac{r}{r-1}-\frac{a}{a-1}\\
&=\frac{b(2a-1)^2(k-1)+s(a-b)^2(a-1)}{(a-1)(a+b-1)(k-1)}-\frac{r}{r-1}-\frac{a}{a-1}.
\end{align*}
For convenience, set $\phi=\frac{b(2a-1)^2(k-1)+s(a-b)^2(a-1)}{(a-1)(a+b-1)(k-1)}$. Clearly, $p(\widetilde{G})=2$ if and only if $\rho\leq0$, if and only if
$$\phi\leq\frac{r}{r-1}+\frac{a}{a-1}.$$
Clearly, $\phi\geq \frac{b(2a-1)^2}{(a-1)(a+b-1)}$ . If $b\geq2$, since $2a-1\geq a+b-1$ we have
$$\phi\geq\frac{2(2a-1)}{a-1}>4\geq\frac{r}{r-1}+\frac{a}{a-1},$$
and thus $p(\widetilde{G})>2$. If $b=1$, then
$$\phi\geq\frac{(2a-1)^2}{a(a-1)}>4\geq\frac{r}{r-1}+\frac{a}{a-1},$$
leading to $p(\widetilde{G})>2$. If $b=0$, then $\phi=\frac{a^2s}{(a-1)(a+s-1)}$ . If $s=0$, then $\phi=0$. If $s=1$, then $\phi=\frac{a}{a-1}$. In both cases, we have $\phi\leq\frac{r}{r-1}+\frac{a}{a-1}$ and thus $p(\widetilde{G})=2$. If $s\geq2$,
the condition $\phi\leq\frac{r}{r-1}+\frac{a}{a-1}$ is equivalent to
$$\frac{1}{r}+\frac{1}{a}+\frac{1}{s-1}\geq1.$$
This completes the proof.
\end{proof}
\begin{cor}\label{cor}
For $r\geq2, k\geq2$ and $p\geq1$, the positive inertia index of $\widetilde{G}=K(r; k; p)$ is 2 if and only if one of the following conditions holds:
\begin{wst}
  \item[{\rm (1)}] $p=1$;
  \item[{\rm (2)}] $p\geq2$ and $k-p\leq1$;
  \item[{\rm (3)}] $p\geq2, k-p\geq2$ and $\frac{1}{r}+\frac{1}{p}+\frac{1}{k-p-1}\geq1$.
\end{wst}
\end{cor}
\begin{proof}
Using mixed two-way switching of mixed graphs we have $\widetilde{G}=K(r; k; p)$ is switching equivalent to the mixed graph $\widetilde{G}=K(r; k; p^i, 0^{-i}, 0^1, 0^{-1}, k-p)$. Thus, the result follows from Lemma \ref{3.8}.
\end{proof}
\begin{lem}\label{3.10}
For $r\geq2, k\geq2, a\geq c, a\geq1$ and $s=k-a-c$, the positive inertia index of $\widetilde{G}=K(r; k; a^i, 0^{-i}, c^{1}, 0^{-1}, s)$ is 2 if and only if one of the following conditions holds:
\begin{wst}
  \item[{\rm (1)}] $a=1, c=0$;
  \item[{\rm (2)}] $a=c=1$ and either $s=0$ or $s=1$;
  \item[{\rm (3)}] $a=c=1, s=2$ and either $r=3$ or $r=4$;
  \item[{\rm (4)}] $a=c=1$, $s=3$ and $r=3$;
  \item[{\rm (5)}] $a=c=1$, $s\geq2$ and $r=2$;
  \item[{\rm (6)}] $a=c=2, s=0$ and $2\leq r\leq4$;
  \item[{\rm (7)}] $a=c=2, s=1$ and $r=2$;
  \item[{\rm (8)}] $a=3, s=0$ and $r=c=2$;
  \item[{\rm (9)}] $a=4, c=r=2$ and $s=0$;
  \item[{\rm (10)}] $a\geq2, c=1$ and $\frac{as-1}{a+s}\leq\frac{1}{r-1}$;
  \item[{\rm (11)}] $a\geq2, c=0$ and either $s=0$ or $s=1$;
  \item[{\rm (12)}] $a\geq2, c=0, s\geq2$ and $\frac{1}{r}+\frac{1}{a}+\frac{1}{s-1}\geq1$.
\end{wst}
\end{lem}
\begin{proof}
The Hermitian adjacency matrix of $\widetilde{G}=K(r; k; a^i, 0^{-i}, c^{1}, 0^{-1}, s)$ can be written as
\begin{align*}
H(\widetilde{G})={\left(
        \begin{array}{ccccc}
        H(K_r) & \mathbf{1} & {\bf 0} & {\bf 0} & {\bf 0} \\[3pt]
        \mathbf{1}^{\top} & 0 & i\mathbf{1}^{\top} & \mathbf{1}^{\top} & {\bf 0} \\[3pt]
        {\bf 0} & -i\mathbf{1} & H(K_a) & {\bf J} & {\bf J} \\[3pt]
        {\bf 0} & \mathbf{1} & {\bf J} & H(K_c) & {\bf J} \\[3pt]
        {\bf 0} & {\bf 0} & {\bf J} & {\bf J} & H(K_s)
        \end{array}
        \right)}.
\end{align*}

{\bf{Case 1.}}\ $a=1$ and $c=0$. The result follows from Lemma \ref{3.8} (1). Thus in this case, the positive inertia index of $K(r; k; a^i, 0^{-i}, c^{1}, 0^{-1}, s)$ must be 2.

{\bf{Case 2.}}\ $a=c=1$. If $s=0$, then
\begin{align*}
H(\widetilde{G})={\left(
        \begin{array}{cccc}
        H(K_r) & \mathbf{1} & {\bf 0} & {\bf 0} \\[3pt]
        \mathbf{1}^{\top} & 0 & i & 1 \\[3pt]
        {\bf 0} & -i & 0 & 1 \\[3pt]
        {\bf 0} & 1 & 1 & 0
        \end{array}
        \right)}.
\end{align*}
Let
\begin{align*}
T_1={\left(
        \begin{array}{cccc}
        I_r & -\frac{1}{r-1}\mathbf{1} & {\bf 0} & {\bf 0} \\[3pt]
        {\bf 0} & 1 & 0 & 0 \\[3pt]
        {\bf 0} & -1 & 1 & 0 \\[3pt]
        {\bf 0} & i & 0 & 1
        \end{array}
        \right)}.
\end{align*}
Then
\begin{align*}
T_1^*H(\widetilde{G})T_1={\left(
        \begin{array}{cccc}
        H(K_r) & {\bf 0} & {\bf 0} & {\bf 0} \\[3pt]
        {\bf 0} & -\frac{r}{r-1} & 0 & 0 \\[3pt]
        {\bf 0} & 0 & 0 & 1 \\[3pt]
        {\bf 0} & 0 & 1 & 0
        \end{array}
        \right)}.
\end{align*}
Thus $p(\widetilde{G})=2$ since $r\geq2$. When $s=1$, then
\begin{align*}
H(\widetilde{G})={\left(
        \begin{array}{ccccc}
        H(K_r) & \mathbf{1} & {\bf 0} & {\bf 0} & {\bf 0} \\[3pt]
        \mathbf{1}^{\top} & 0 & i & 1 & 0 \\[3pt]
        {\bf 0} & -i & 0 & 1 & 1 \\[3pt]
        {\bf 0} & 1 & 1 & 0 & 1 \\[3pt]
        {\bf 0} & 0 & 1 & 1 & 0
        \end{array}
        \right)}.
\end{align*}
Let
\begin{align*}
T_2={\left(
        \begin{array}{ccccc}
        I_r & -\frac{1}{r-1}\mathbf{1} & {\bf 0} & {\bf 0} & {\bf 0} \\[3pt]
        {\bf 0} & 1 & 0 & 0 & 0 \\[3pt]
        {\bf 0} & 0 & 1 & 0 & 0 \\[3pt]
        {\bf 0} & 0 & -1 & 1 & 0 \\[3pt]
        {\bf 0} & -1 & -1 & 0 & 1
        \end{array}
        \right)}.
\end{align*}
Then
\begin{align*}
T_2^*H(\widetilde{G})T_2={\left(
        \begin{array}{ccccc}
        H(K_r) & {\bf 0} & {\bf 0} & {\bf 0} & {\bf 0} \\[3pt]
        {\bf 0} & -\frac{r}{r-1} & i-1 & 0 & 0 \\[3pt]
        {\bf 0} & -i-1 & -2 & 0 & 0 \\[3pt]
        {\bf 0} & 0 & 0 & 0 & 1 \\[3pt]
        {\bf 0} & 0 & 0 & 1 & 0
        \end{array}
        \right)}.
\end{align*}
Thus $p(\widetilde{G})=2$ if and only if
\begin{align*}
{\left(
        \begin{array}{cc}
        -\frac{r}{r-1} & i-1 \\[3pt]
        -i-1 & -2
        \end{array}
        \right)}
\end{align*}
has no positive eigenvalue, if and only if $r\geq2$. When $s\geq2$, let
\begin{align*}
T_3={\left(
        \begin{array}{ccccc}
        I_r & -\frac{1}{r-1}\mathbf{1} & {\bf 0} & {\bf 0} & {\bf 0} \\[3pt]
        {\bf 0} & 1 & 0 & 0 & {\bf 0} \\[3pt]
        {\bf 0} & 0 & 1 & 0 & {\bf 0} \\[3pt]
        {\bf 0} & 0 & 0 & 1 & {\bf 0} \\[3pt]
        {\bf 0} & {\bf 0} & -\frac{1}{s-1}\mathbf{1} & -\frac{1}{s-1}\mathbf{1} & I_s
        \end{array}
        \right)}.
\end{align*}
Then
\begin{align*}
T_3^*H(\widetilde{G})T_3={\left(
        \begin{array}{ccccc}
        H(K_r) & {\bf 0} & {\bf 0} & {\bf 0} & {\bf 0} \\[3pt]
        {\bf 0} & -\frac{r}{r-1} & i & 1 & {\bf 0} \\[3pt]
        {\bf 0} & -i & -\frac{s}{s-1} & -\frac{1}{s-1} & {\bf 0} \\[3pt]
        {\bf 0} & 1 & -\frac{1}{s-1} & -\frac{s}{s-1} & {\bf 0} \\[3pt]
        {\bf 0} & {\bf 0} & {\bf 0} & {\bf 0} & H(K_s)
        \end{array}
        \right)}.
\end{align*}
If $r\geq3$, since the determinant of
\begin{align}
{\left(
        \begin{array}{ccc}
        -\frac{r}{r-1} & i & 1 \\[3pt]
        -i & -\frac{s}{s-1} & -\frac{1}{s-1} \\[3pt]
        1 & -\frac{1}{s-1} & -\frac{s}{s-1}
        \end{array}
        \right)}
\end{align}
is $\frac{r(s-1)-2s}{(r-1)(s-1)}$. If $s\geq4$, then $\frac{r(s-1)-2s}{(r-1)(s-1)}\geq\frac{s-3}{(r-1)(s-1)}>0$. Thus the positive inertia index of matrix (3.1) is at least 1, and thus $p(\widetilde{G})\geq3$. If $s=3$ and $r=3$, matrix (3.1) has eigenvalues 0, $-\frac{3}{2}$ and $-3$ and thus $p(\widetilde{G})=2$. If $s=3$ and $r>3$, then $\frac{r(s-1)-2s}{(r-1)(s-1)}=\frac{r-3}{r-1}>0$. Thus the positive inertia index of matrix (3.1) is at least 1, and thus $p(\widetilde{G})\geq3$. If $s=2$ and $r>4$, then $\frac{r(s-1)-2s}{(r-1)(s-1)}=\frac{r-4}{r-1}>0$. Thus the positive inertia index of matrix (3.1) is at least 1, and thus $p(\widetilde{G})\geq3$. If $s=2$ and $r=4$, matrix (3.1) has eigenvalues 0, $\frac{-8+\sqrt{7}}{3}$ and $\frac{-8-\sqrt{7}}{3}$. If $s=2$ and $r=3$, matrix (3.1) has eigenvalues $-3.5878$, $-1.8363$ and $-0.0759$. In both cases, we have the positive inertia index of matrix (3.1) is 0, and thus $p(\widetilde{G})=2$.

If $r=2$, the characteristic polynomial of matrix (3.1) is $f(\lambda)=\lambda^3+\frac{2(2s-1)}{s-1}\lambda^2+\frac{3(s+1)}{(s-1)}\lambda+\frac{2}{s-1}$. We have $f'(\lambda)=3\lambda^2+\frac{4(2s-1)}{s-1}\lambda+\frac{3(s+1)}{(s-1)}$. It is easy to see that $f'(\lambda)>0$ for $\lambda\geq0$. Thus $f(\lambda)$ is strictly increasing in $[0, +\infty)$. Since $f(0)=\frac{2}{s-1}>0$, we have matrix (3.1) has no positive eigenvalue, and thus $p(\widetilde{G})=2$.

{\bf{Case 3.}}\ $a\geq2$. If $k=a=2$, then by Lemma \ref{lem3.3},
$$2=p(K_r)+p(K_k)\leq p(\widetilde{G})\leq p(K_{r+1})+p(\widetilde{K_{k+1}})=2,$$
and thus $p(\widetilde{G})=2$. Suppose $k\geq3$. Similarly, as in the proof of Case 3 of Lemma \ref{3.8}, set
\begin{align*}
T_4={\left(
        \begin{array}{ccc}
        I_r & {\bf 0} & {\bf 0} \\[3pt]
        {\bf 0} & 1 & {\bf 0} \\[3pt]
        {\bf 0} & {\bf 0} & P
        \end{array}
        \right)}.
\end{align*}

Then $T_4^*H(\widetilde{G})T_4$ is the following matrix:
\begin{align*}
{\left(
        \begin{array}{cccccccccc}
        H(K_r) & \mathbf{1} & {\bf 0} & {\bf 0} & {\bf 0} & \cdots & {\bf 0} & {\bf 0} & \cdots & {\bf 0} \\[3pt]
        \mathbf{1}^{\top} & 0 & i\mathbf{1}^{\top} & \frac{(a-1)-ai}{a-1} & \frac{(a-1)-ai}{a} & \cdots & \frac{(a-1)-ai}{a+c-2} & \frac{-c-ai}{a+c-1} & \cdots & \frac{-c-ai}{k-2} \\[3pt]
        {\bf 0} & -i\mathbf{1} & H(K_a) & 0 & 0 & \cdots & 0 & 0 & \cdots & 0 \\[3pt]
        {\bf 0} & \frac{(a-1)+ai}{a-1} & 0 & \frac{-a}{a-1} & 0 & \cdots & 0 & 0 & \cdots & 0 \\[3pt]
        {\bf 0} & \frac{(a-1)+ai}{a} & 0 & 0 & \frac{-(a+1)}{a} & \cdots & 0 & 0 & \cdots & 0 \\[3pt]
        \vdots & \vdots & \vdots & \vdots & \vdots & \ddots & \vdots & \vdots & \cdots & \vdots \\[3pt]
        {\bf 0} & \frac{(a-1)+ai}{a+c-2} & 0 & 0 & 0 & \cdots & \frac{-(a+c-1)}{a+c-2} & 0 & \cdots & 0 \\[3pt]
        {\bf 0} & \frac{-c+ai}{a+c-1} & 0 & 0 & 0 & \cdots & 0 & \frac{-(a+c)}{a+c-1} & \cdots & 0 \\[3pt]
        \vdots & \vdots & \vdots & \vdots & \vdots & \ddots & \vdots & \vdots & \cdots & \vdots \\[3pt]
        {\bf 0} & \frac{-c+ai}{k-2} & 0 & 0 & 0 & \cdots & 0 & 0 & \cdots & \frac{-(k-1)}{k-2}
        \end{array}
        \right)}.
\end{align*}
Let
\begin{align*}
T_5={\left(
        \begin{array}{cccccccccc}
        I_r & -\frac{1}{r-1}\mathbf{1} & {\bf 0} & {\bf 0} & {\bf 0} & \cdots & {\bf 0} & {\bf 0} & \cdots & {\bf 0} \\[3pt]
        {\bf 0} & 1 & 0 & 0 & 0 & \cdots & 0 & 0 & \cdots & 0 \\[3pt]
        {\bf 0} & \frac{i}{a-1}\mathbf{1} & I_a & 0 & 0 & \cdots & 0 & 0 & \cdots & 0 \\[3pt]
        {\bf 0} & \frac{(a-1)+ai}{a} & 0 & 1 & 0 & \cdots & 0 & 0 & \cdots & 0 \\[3pt]
        {\bf 0} & \frac{(a-1)+ai}{a+1} & 0 & 0 & 1 & \cdots & 0 & 0 & \cdots & 0 \\[3pt]
        \vdots & \vdots & \vdots & \vdots & \vdots & \ddots & \vdots & \vdots & \cdots & \vdots \\[3pt]
        {\bf 0} & \frac{(a-1)+ai}{a+c-1} & 0 & 0 & 0 & \cdots & 1 & 0 & \cdots & 0 \\[3pt]
        {\bf 0} & \frac{-c+ai}{a+c} & 0 & 0 & 0 & \cdots & 0 & 1 & \cdots & 0 \\[3pt]
        \vdots & \vdots & \vdots & \vdots & \vdots & \ddots & \vdots & \vdots & \cdots & \vdots \\[3pt]
        {\bf 0} & \frac{-c+ai}{k-1} & 0 & 0 & 0 & \cdots & 0 & 0 & \cdots & 1
        \end{array}
        \right)}.
\end{align*}
Then $T_5^*T_4^*H(\widetilde{G})T_4T_5$ is the following matrix:
\begin{align*}
{\left(
        \begin{array}{cccccccccc}
        H(K_r) & {\bf 0} & {\bf 0} & {\bf 0} & {\bf 0} & \ldots & {\bf 0} & {\bf 0} & \ldots & {\bf 0} \\[3pt]
        {\bf 0} & \xi & 0 & 0 & 0 & \ldots & 0 & 0 & \ldots & 0 \\[3pt]
        {\bf 0} & 0 & H(K_a) & 0 & 0 & \ldots & 0 & 0 & \ldots & 0 \\[3pt]
        {\bf 0} & 0 & 0 & \frac{-a}{a-1} & 0 & \ldots & 0 & 0 & \ldots & 0 \\[3pt]
        {\bf 0} & 0 & 0 & 0 & \frac{-(a+1)}{a} & \ldots & 0 & 0 & \ldots & 0 \\[3pt]
        \vdots & \vdots & \vdots & \vdots & \vdots & \ddots & \vdots & \vdots & \ldots & \vdots \\[3pt]
        {\bf 0} & 0 & 0 & 0 & 0 & \ldots & \frac{-(a+c-1)}{a+c-2} & 0 & \ldots & 0 \\[3pt]
        {\bf 0} & 0 & 0 & 0 & 0 & \ldots & 0 & \frac{-(a+c)}{a+c-1} & \ldots & 0 \\[3pt]
        \vdots & \vdots & \vdots & \vdots & \vdots & \ddots & \vdots & \vdots & \ldots & \vdots \\[3pt]
        {\bf 0} & 0 & 0 & 0 & 0 & \ldots & 0 & 0 & \ldots & \frac{-(k-1)}{k-2}
        \end{array}
        \right)},
\end{align*}
where
\begin{align*}
\xi&=[(a-1)^2+a^2]\left(\frac{1}{a-1}-\frac{1}{a+c-1}\right)+(a^2+c^2)\left(\frac{1}{a+c-1}-\frac{1}{k-1}\right)-\frac{r}{r-1}-\frac{a}{a-1}\\
&=\frac{c[(a-1)^2+a^2](k-1)+s(a-1)(a^2+c^2)}{(a-1)(a+c-1)(k-1)}-\frac{r}{r-1}-\frac{a}{a-1}.
\end{align*}
For convenience, set $\gamma=\frac{c[(a-1)^2+a^2](k-1)+s(a-1)(a^2+c^2)}{(a-1)(a+c-1)(k-1)}$. Clearly, $p(\widetilde{G})=2$ if and only if $\xi\leq0$, if and only if
$$\gamma\leq\frac{r}{r-1}+\frac{a}{a-1}.$$

If $c\geq4$, we have
$$
\gamma\geq\frac{c[(a-1)^2+a^2]}{(a-1)(a+c-1)}\geq\frac{4[(a-1)^2+a^2]}{(a-1)(2a-1)}=\frac{4(2a^2-2a+1)}{2a^2-3a+1}>4\geq\frac{r}{r-1}+\frac{a}{a-1},
$$
and thus $p(\widetilde{G})>2$.

If $c=3$ and $a\geq4$, then we have
$$
\gamma\geq\frac{c[(a-1)^2+a^2]}{(a-1)(a+c-1)}=\frac{3[(a-1)^2+a^2]}{(a-1)(a+2)}>4\geq\frac{r}{r-1}+\frac{a}{a-1},
$$
and thus $p(\widetilde{G})>2$.

If $c=a=3$, then $\gamma=\frac{39}{10}+\frac{18s}{5(5+s)}$, where $k=a+c+s$. If $\gamma\leq\frac{r}{r-1}+\frac{a}{a-1}=\frac{r}{r-1}+\frac{3}{2}$, then we have $\frac{12}{5}+\frac{18s}{5(5+s)}\leq\frac{r}{r-1}$. Hence, we obtain $2<\frac{r}{r-1}$, i.e. $r<2$, a contradiction. Thus, $p(\widetilde{G})>2$.

If $c=2$ and $a\geq4$, we now consider three cases. If $r\geq3$, then $\widetilde{R_1}$ (see Fig. 4) is an induced subgraph of $\widetilde{G}$. If $s\geq1$, then $\widetilde{R_2}$ (see Figure. 4) is an induced subgraph of $\widetilde{G}$. By direct calculation, we have $p(\widetilde{R_1})=p(\widetilde{R_2})=3$. Hence, $p(\widetilde{G})=2$ implies $r=2$ and $s=0$. Let $\gamma=\frac{c[(a-1)^2+a^2]}{(a-1)(a+c-1)}=\frac{2(2a^2-2a+1)}{a^2-1}\leq\frac{r}{r-1}+\frac{a}{a-1}=2+\frac{a}{a-1}$. We have $a=4$, $r=2$ and $s=0$. Hence, in this case, $p(\widetilde{G})=2$ if and only if $a=4, r=2$ and $s=0$.
\begin{figure}[h!]
\begin{center}
\psfrag{1}{$v$}\psfrag{2}{$\widetilde{R_1}$}\psfrag{3}{$\widetilde{R_2}$}\psfrag{a}{Figure 4: Forbidden subgraph $\widetilde{R_1}$ and $\widetilde{R_2}$.}
\includegraphics[width=120mm]{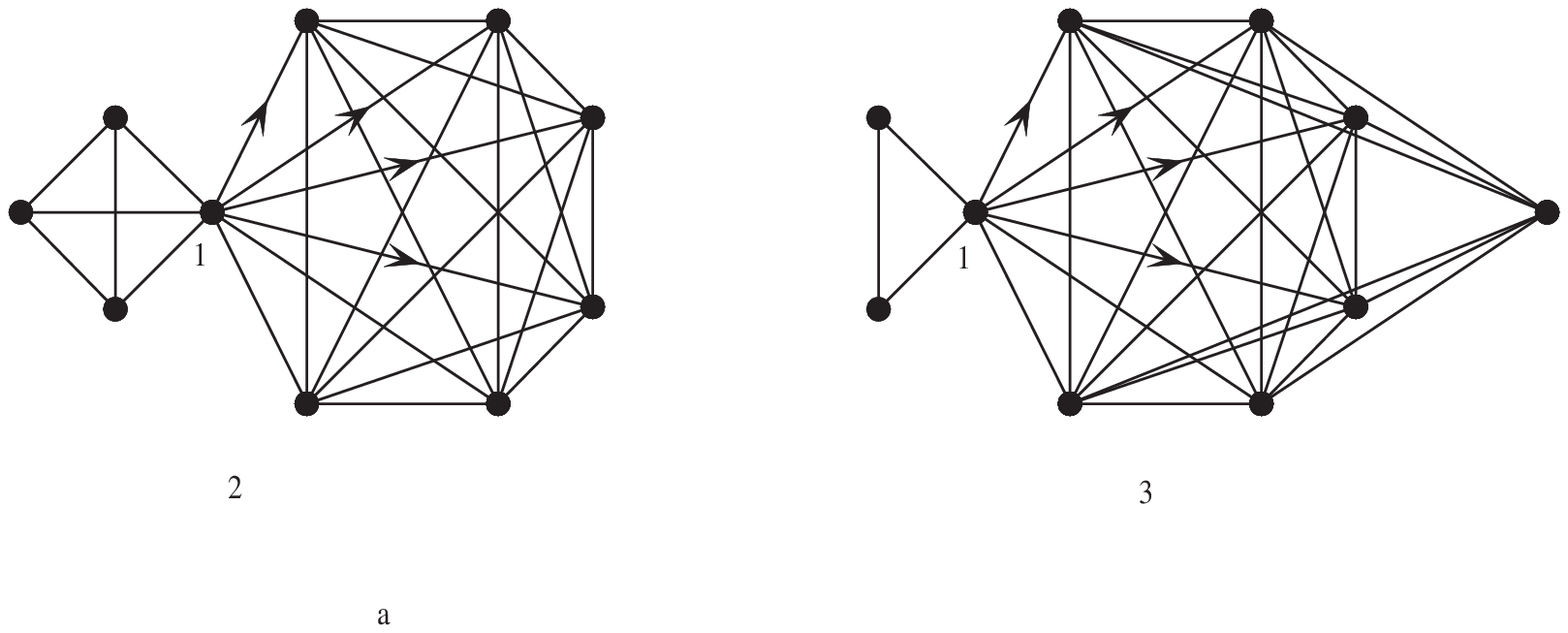} \\
\end{center}
\end{figure}

If $c=2$ and $a=3$, then $\gamma=\frac{13}{4}+\frac{13s}{4(4+s)}$, where $k=a+c+s$. If $\gamma\leq\frac{r}{r-1}+\frac{a}{a-1}=\frac{r}{r-1}+\frac{3}{2}$, then we have $(3+4s)r\leq7+5s$. Note that $r\geq2$, we obtain $2(3+4s)\leq7+5s$. Thus, we have $s=0$. In addition, $3r\leq7$ i.e. $r=2$. Thus, in this case, $p(\widetilde{G})=2$ if and only if $r=2$ and $s=0$.

If $c=2$ and $a=2$, then $\gamma=\frac{10(3+s)+8s}{3(3+s)}$, where $k=a+c+s$. If $\gamma\leq\frac{r}{r-1}+\frac{a}{a-1}=\frac{r}{r-1}+2$, then we have $(3+9s)r\leq12+12s$. Note that $r\geq2$, we obtain $2(3+9s)\leq12+12s$. Thus, we have $s=0$ or 1. If $s=0$, we have $2\leq r\leq4$. If $s=1$, we obtain $r=2$. Thus, in this case, $p(\widetilde{G})=2$ if and only if (i) $s=0$ and $2\leq r\leq4$ or (ii) $s=1$ and $r=2$.

If $c=1$, then $\gamma=\frac{[(a-1)^2+a^2](k-1)+s(a-1)(a^2+1)}{a(a-1)(k-1)}=\frac{[(a-1)^2+a^2](a+s)+s(a-1)(a^2+1)}{a(a-1)(a+s)}$, where $k=a+c+s$. The condition $\gamma\leq\frac{r}{r-1}+\frac{a}{a-1}$ is equivalent to $\frac{as-1}{a+s}\leq\frac{1}{r-1}$.

If $c=0$, then $\gamma=\frac{a^2s}{(a-1)(k-1)}=\frac{a^2s}{(a-1)(a+s-1)}$, where $k=a+c+s$. If $s=0$, then $\gamma=0$. If $s=1$, then $\gamma=\frac{a}{a-1}$. In both cases, we have $\gamma\leq\frac{r}{r-1}+\frac{a}{a-1}$ and thus $p(\widetilde{G})=2$. If $s\geq2$, the condition $\gamma\leq\frac{r}{r-1}+\frac{a}{a-1}$ is equivalent to
$$\frac{1}{r}+\frac{1}{a}+\frac{1}{s-1}\geq1.$$
This completes the proof.
\end{proof}
\begin{lem}\label{lem3.7}
Let $\widetilde{F_1}$ be a connected mixed graph with $p(\widetilde{F_1})=1$. Let $\widetilde{F_2}$ be a mixed graph obtained from $\widetilde{F_1}$ by adding a new vertex $v$ and adding some mixed edges between $v$ and some vertices of $\widetilde{F_1}$. If $rk(\widetilde{F_2})=rk(\widetilde{F_1})+1$ and $p(\widetilde{F_2})=2$, then $\widetilde{F_1}$ is switching equivalent to some complete $k$-partite graph $K_{n_1', n_2',\ldots, n_k'}$ and $\widetilde{F_2}$ is switching equivalent to some complex unit gain graph $K(0; n_1'', n_2'',\ldots, n_k''; a_*^i, b_*^{-i}, c_*^{1}, d_*^{-1}, k-a_*-b_*-c_*-d_*)$.
\end{lem}
\begin{proof}
Since $p(\widetilde{F_1})=1$, by Lemma \ref{lem1.9}, we have $\widetilde{F_1}$ is switching equivalent to a complete multipartite graph or some $\overrightarrow{C_3}(t_1, t_2, t_3)$. In view of Lemma \ref{rank}, we have $rk(\overrightarrow{C_3}(t_1, t_2, t_3))=2$. If $\widetilde{F_1}$ is switching equivalent to some $\overrightarrow{C_3}(t_1, t_2, t_3)$, then combining $rk(\overrightarrow{C_3}(t_1, t_2, t_3))=2$ and $rk(\widetilde{F_2})=rk(\widetilde{F_1})+1$, we have $rk(\widetilde{F_2})=3$. By Lemma \ref{ev}, we have $T_{\widetilde{F_2}}$ is an even triangle. Since $\widetilde{F_1}$ is a subgraph of $\widetilde{F_2}$, hence $\widetilde{F_2}$ contains an odd triangle. By Lemma \ref{si}, odd triangle is not equivalent to even triangle. Thus, we obtain $T_{\widetilde{F_2}}$ is not an even triangle, a contradiction. Hence, we have $\widetilde{F_1}$ is switching equivalent to a complete multipartite graph.

Without loss of generality, assume that $\widetilde{F_1}$ is switching equivalent to $K_{n_1', n_2',\ldots, n_k'}$. Then $\widetilde{F_2}$ is switching equivalent to a complex unit gain graph $F_2^{\varphi}$ with $F_2^{\varphi}-v=K_{n_1', n_2',\ldots, n_k'}$, where $\varphi(\overrightarrow{E(F_2)})\subseteq\{\pm1, \pm i\}$. Now we consider $F_2^{\varphi}$. Let $V_1, V_2, \ldots, V_k$ be the partition classes of $V(K_{n_1', n_2',\ldots, n_k'})$ such that $|V_j|=n_j'$ for $j=1, 2, \ldots, k$. Since $rk(\widetilde{F_2})=rk(\widetilde{F_1})+1$, we have $v$ is not an isolated vertex. In the following, we will prove the following two facts.

\begin{fact}
Let $1\leq j\leq k$. If there exists one edge connecting $v$ and a vertex of $V_j$ in $F_2$, then $v$ is adjacent to all vertices of $V_j$ in $F_2$.
\end{fact}
\noindent{\bf Proof of Fact 1} Suppose on the contrary, without loss of generality, let ${u, w}\in V_j$ such that $vu\in E(F_2)$ and $vw\notin E(F_2)$. Without loss of generality, choose arbitrary vertex $v_t\in V_t$ for $1\leq t\leq k, t\neq j$. The Hermitian adjacency matrix of the complex unit gain graph induced by vertex subset $$\{v, u, w\}\bigcup\Big(\bigcup_{1\leq t\leq k, t\neq j}\{v_t\}\Big)$$ of $F_2^{\varphi}$ can be written as
\begin{align*}
N={\left(
        \begin{array}{cccc}
        0 & 0 & \epsilon & \eta^{*} \\[3pt]
        0 & 0 & 0 & \mathbf{1}^{\top} \\[3pt]
        \overline{\epsilon} & 0 & 0 & \mathbf{1}^{\top} \\[3pt]
        \eta & \mathbf{1} & \mathbf{1} & \mathbf{J}_{k-1}-\mathbf{I}_{k-1}
        \end{array}
        \right)},
\end{align*}
where $\epsilon\in\{\pm1, \pm i\}$. Set
\begin{align*}
Q_1={\left(
        \begin{array}{cccc}
        1 & 0 & 0 & \mathbf{0} \\[3pt]
        0 & 1 & 0 & \mathbf{0} \\[3pt]
        0 & -1 & 1 & \mathbf{0} \\[3pt]
        \mathbf{0} & \mathbf{0} & \mathbf{0} & \mathbf{I}_{k-1}
        \end{array}
        \right)},
\end{align*}
then $Q_1^{*}NQ_1$ is the following matrix:
\begin{align*}
{\left(
        \begin{array}{cccc}
        0 & -\epsilon & \epsilon & \eta^{*} \\[3pt]
        -\overline{\epsilon} & 0 & 0 & \mathbf{0} \\[3pt]
        \overline{\epsilon} & 0 & 0 & \mathbf{1}^{\top} \\[3pt]
        \eta & \mathbf{0} & \mathbf{1} & \mathbf{J}_{k-1}-\mathbf{I}_{k-1}
        \end{array}
        \right)}.
\end{align*}
Note that the equation
\begin{align*}
\left(
        \begin{array}{cc}
        0 & \mathbf{1}^{\top} \\[3pt]
        \mathbf{1} & \mathbf{J}_{k-1}-\mathbf{I}_{k-1}
        \end{array}
        \right)X
        =\left(\begin{array}{c}
       \overline{\epsilon} \\
       \eta \\
       \end{array}
       \right)
\end{align*}
has a solution, say $\left(\begin{array}{c}
       \mu \\
       \vartheta \\
       \end{array}
       \right).$
Let
\begin{align*}
Q_2={\left(
        \begin{array}{cccc}
        1 & 0 & 0 & \mathbf{0} \\[3pt]
        0 & 1 & 0 & \mathbf{0} \\[3pt]
        -\mu & 0 & 1 & \mathbf{0} \\[3pt]
        -\vartheta & \mathbf{0} & \mathbf{0} & \mathbf{I}_{k-1}
        \end{array}
        \right)}.
\end{align*}
Then $Q_2^*Q_1^*NQ_1Q_2$ is the following matrix:
\begin{align*}
{\left(
        \begin{array}{cccc}
        -\epsilon\mu-\eta^{*}\vartheta & -\epsilon & 0 & \mathbf{0} \\[3pt]
        -\overline{\epsilon} & 0 & 0 & \mathbf{0} \\[3pt]
        0 & 0 & 0 & \mathbf{1}^{\top} \\[3pt]
        \mathbf{0}& \mathbf{0} & \mathbf{1} & \mathbf{J}_{k-1}-\mathbf{I}_{k-1}
        \end{array}
        \right)}.
\end{align*}
Thus
\begin{align*}
rank(N)&=rank(Q_2^*Q_1^*NQ_1Q_2)\\
&=rank{\left(
        \begin{array}{cc}
        -\epsilon\mu-\eta^{*}\vartheta & -\epsilon \\[3pt]
        -\overline{\epsilon} & 0
        \end{array}
        \right)}+rank{\left(
        \begin{array}{cc}
        0 & \mathbf{1}^{\top} \\[3pt]
        \mathbf{1} & \mathbf{J}_{k-1}-\mathbf{I}_{k-1}
        \end{array}
        \right)}\\
&=2+k\\
&=2+rk(\widetilde{F}_1).
\end{align*}
Since $rk(\widetilde{F}_2)=rk(F_2^{\varphi})\geq rank(N)$, we have $rk(\widetilde{F}_2)\geq 2+rk(\widetilde{F}_1)$, a contradiction.\qed
\begin{fact}
$\forall 1\leq j\leq k$, $A(F_2^{\varphi})_{vu}=A(F_2^{\varphi})_{vw}$ for all $\{u, w\}\subseteq V_j$.
\end{fact}
\noindent{\bf Proof of Fact 2} Let $1\leq j\leq k$. If there are no edges between $v$ and $V_j$, then it is easy to see that $A(F_2^{\varphi})_{vu}=A(F_2^{\varphi})_{vw}=0$ for all $\{u, w\}\subseteq V_j$. If there are edges between $v$ and $V_j$, then by $\mathbf{Fact \ 1}$, we have $vu$ and $vw$ are edges of $F_2$ for all $\{u, w\}\subseteq V_j$. Suppose $A(F_2^{\varphi})_{vu}\neq A(F_2^{\varphi})_{vw}$ for some $\{u, w\}\subseteq V_j$. Without loss of generality, choose arbitrary vertex $v_t\in V_t$ for $1\leq t\leq k$ and $t\neq j$. The Hermitian adjacency matrix of the complex unit gain graph induced by vertex subset
$$\{v, u, w\}\bigcup(\bigcup_{1\leq t\leq k, t\neq j}\{v_t\})$$
of $F_2^{\varphi}$ can be written as
\begin{align*}
M={\left(
        \begin{array}{cccc}
        0 & A(F_2^{\varphi})_{vu} & A(F_2^{\varphi})_{vw} & \zeta^{*} \\[3pt]
        \overline{A(F_2^{\varphi})_{vu}} & 0 & 0 & \mathbf{1}^{\top} \\[3pt]
        \overline{A(F_2^{\varphi})_{vw}} & 0 & 0 & \mathbf{1}^{\top} \\[3pt]
        \zeta & \mathbf{1} & \mathbf{1} & \mathbf{J}_{k-1}-\mathbf{I}_{k-1}
        \end{array}
        \right)}.
\end{align*}
Set
\begin{align*}
Q_3={\left(
        \begin{array}{cccc}
        1 & 0 & 0 & \mathbf{0} \\[3pt]
        0 & 1 & 0 & \mathbf{0} \\[3pt]
        0 & -1 & 1 & \mathbf{0} \\[3pt]
        \mathbf{0} & \mathbf{0} & \mathbf{0} & \mathbf{I}_{k-1}
        \end{array}
        \right)},
\end{align*}
then $Q_3^{*}MQ_3$ is the following matrix:
\begin{align*}
{\left(
        \begin{array}{cccc}
        0 & A(F_2^{\varphi})_{vu}-A(F_2^{\varphi})_{vw} & A(F_2^{\varphi})_{vw} & \zeta^{*} \\[3pt]
        \overline{A(F_2^{\varphi})_{vu}}-\overline{A(F_2^{\varphi})_{vw}} & 0 & 0 & \mathbf{0} \\[3pt]
        \overline{A(F_2^{\varphi})_{vw}} & 0 & 0 & \mathbf{1}^{\top} \\[3pt]
        \zeta & \mathbf{0} & \mathbf{1} & \mathbf{J}_{k-1}-\mathbf{I}_{k-1}
        \end{array}
        \right)}.
\end{align*}
Note that the equation
\begin{align*}
\left(
        \begin{array}{cc}
        0 & \mathbf{1}^{\top} \\[3pt]
        \mathbf{1} & \mathbf{J}_{k-1}-\mathbf{I}_{k-1}
        \end{array}
        \right)X
        =\left(\begin{array}{c}
       \overline{A(F_2^{\varphi})_{vw}} \\
       \zeta \\
       \end{array}
       \right)
\end{align*}
has a solution, say $\left(\begin{array}{c}
       \nu \\
       \delta \\
       \end{array}
       \right).$
Let
\begin{align*}
Q_4={\left(
        \begin{array}{cccc}
        1 & 0 & 0 & \mathbf{0} \\[3pt]
        0 & 1 & 0 & \mathbf{0} \\[3pt]
        -\nu & 0 & 1 & \mathbf{0} \\[3pt]
        -\delta & \mathbf{0} & \mathbf{0} & \mathbf{I}_{k-1}
        \end{array}
        \right)}.
\end{align*}
Then $Q_4^*Q_3^*NQ_3Q_4$ is the following matrix:
\begin{align*}
{\left(
        \begin{array}{cccc}
        -\nu A(F_2^{\varphi})_{vw}-\zeta^{*}\delta & A(F_2^{\varphi})_{vu}-A(F_2^{\varphi})_{vw} & 0 & \mathbf{0} \\[3pt]
        \overline{A(F_2^{\varphi})_{vu}}-\overline{A(F_2^{\varphi})_{vw}} & 0 & 0 & \mathbf{0} \\[3pt]
        0 & 0 & 0 & \mathbf{1}^{\top} \\[3pt]
        \mathbf{0}& \mathbf{0} & \mathbf{1} & \mathbf{J}_{k-1}-\mathbf{I}_{k-1}
        \end{array}
        \right)}.
\end{align*}
Since $A(F_2^{\varphi})_{vu}\neq A(F_2^{\varphi})_{vw}$, we obtain
\begin{align*}
rank(M)&=rank(Q_4^*Q_3^*MQ_3Q_4)\\
&=rank{\left(
        \begin{array}{cc}
        -\nu A(F_2^{\varphi})_{vw}-\zeta^{*}\delta & A(F_2^{\varphi})_{vu}-A(F_2^{\varphi})_{vw} \\[3pt]
        \overline{A(F_2^{\varphi})_{vu}}-\overline{A(F_2^{\varphi})_{vw}} & 0
        \end{array}
        \right)}+rank{\left(
        \begin{array}{cc}
        0 & \mathbf{1}^{\top} \\[3pt]
        \mathbf{1} & \mathbf{J}_{k-1}-\mathbf{I}_{k-1}
        \end{array}
        \right)}\\
&=2+k\\
&=2+rk(\widetilde{F_1}).
\end{align*}
Since $rk(\widetilde{F_2})=rk(F_2^{\varphi})\geq rk(M)$, we have $rk(\widetilde{F_2})\geq 2+rk(\widetilde{F_1})$, a contradiction. This completes the proof of Fact 2.\qed

Combining Facts 1 and 2, we have $\widetilde{F_2}$ is switching equivalent to some complex unit gain graph
$$K(0; n_1'', n_2'',\ldots, n_k''; a_*^i, b_*^{-i}, c_*^{1}, d_*^{-1}, k-a_*-b_*-c_*-d_*).$$
\end{proof}
\subsection{\normalsize Proof for Theorem \ref{thm1.2}}
Now, we are ready to prove the second main result of this paper.

\noindent{\bf Proof of Theorem \ref{thm1.2}}\ \ \textit{``Sufficiency''.} If $\widetilde{G}$ is a mixed graph switching equivalent to a mixed graph satisfying (i), then by Lemmas \ref{lem1.9} and \ref{lem3.2}, the result is naturally obtained. If $\widetilde{G}$ is a mixed graph switching equivalent to a mixed graph satisfying (ii), then by Lemma \ref{twin}, $T_{\widetilde{G}}$ is switching equivalent to $K(r; k; p)$ satisfying one of the three conditions of Corollary \ref{cor}. By Corollary \ref{cor}, we have $p(K(r; k; p))=2$. In view of Lemmas \ref{twin} and \ref{add}, we have $p(\widetilde{G})=p(T_{\widetilde{G}})=p(K(r; k; p))=2$. If $\widetilde{G}$ is a mixed graph switching equivalent to a mixed graph satisfying (iii), then by Lemma \ref{twin}, $T_{\widetilde{G}}$ is switching equivalent to $K(2; k; 1^i, 1^{-i}, 0^{1}, 0^{-1}, s)$ satisfying the condition (2) of Lemma \ref{3.8}. By Lemma \ref{3.8}, we have $p(K(2; k; 1^i, 1^{-i}, 0^{1}, 0^{-1}, s)=2$. In view of Lemmas \ref{twin} and \ref{add}, we have $p(\widetilde{G})=p(T_{\widetilde{G}})=p(K(r; k; a^i, b^{-i}, 0^{1}, 0^{-1}, s)=2$. If $\widetilde{G}$ is a mixed graph switching equivalent to a mixed graph satisfying (iv), then by Lemma \ref{twin}, $T_{\widetilde{G}}$ is switching equivalent to $K(r; k; a^i, 0^{-i}, c^{1}, 0^{-1}, s)$ satisfying one of the conditions (2)-(10) of Lemma \ref{3.10}. By Lemma \ref{3.10}, we have $p(K(r; k; a^i, 0^{-i}, c^{1}, 0^{-1}, s))=2$. In view of Lemmas \ref{twin} and \ref{add}, we have $p(\widetilde{G})=p(T_{\widetilde{G}})=p(K(r; k; a^i, 0^{-i}, c^{1}, 0^{-1}, s))=2$.

\textit{``Necessity''.} We proceed by establishing some claims.
\begin{claim}
$\widetilde{G}-v$ has exactly two components.
\end{claim}
\noindent{\bf Proof of Claim 1.}\
Suppose that $\widetilde{G}-v=\widetilde{G_1}\cup\widetilde{G_2}\cup\ldots\cup\widetilde{G_t}$ is the disjoint union of different components of
$\widetilde{G}-v$. Since $v$ is a cut vertex of $\widetilde{G}$, we have $t\geq2$. Since $\widetilde{G}$ has no pendant vertices, we have $\widetilde{G_i}$ has at least one edge for all $i=1, 2,\ldots, t$. Hence, $p(\widetilde{G_i})\geq1$ for all $i=1, 2,\ldots, t$. From $2=p(\widetilde{G})\geq p(\widetilde{G}-v)=p(\widetilde{G}_1)+\ldots+p(\widetilde{G}_t)$ it follows that $t=2$. Hence, $\widetilde{G}-v$ has exactly two components $\widetilde{G_1}$ and $\widetilde{G_2}$.
\qed
\begin{claim}
$p(\widetilde{G_i})=1$ for each $i=1, 2$.
\end{claim}
\noindent{\bf Proof of Claim 2.}\
Noting that $p(\widetilde{G_1})\geq1$ and $p(\widetilde{G_2})\geq1$, then from
\begin{align*}
2=p(\widetilde{G})\geq p(\widetilde{G_1})+p(\widetilde{G_2})\geq2
\end{align*}
we have $p(\widetilde{G_i})=1$ for each $i=1, 2$.
\qed

By Claim 2 and Lemma \ref{lem1.9}, we have $\widetilde{G_1}$ (resp. $\widetilde{G_2}$) is switching equivalent to a complete multipartite graph or some $\overrightarrow{C_3}(t_1, t_2, t_3)$. Suppose, without loss of generality, that $p(\widetilde{G_1}+v)\leq p(\widetilde{G_2}+v)$.
\begin{claim}
$p(\widetilde{G_1}+v)=1$.
\end{claim}
\noindent{\bf Proof of Claim 3.}\
We distinguish the following three cases.
\begin{Case}
There exists $\widetilde{G_i}$ such that $rk(\widetilde{G_i}+v)=rk(\widetilde{G_i})+2$.
\end{Case}
By Lemma \ref{lem1.6} (1), we have $p(\widetilde{G})=p(\widetilde{G}-v)+1=p(\widetilde{G_1})+p(\widetilde{G_2})+1=3$, a contradiction.

\begin{Case}
There exists $\widetilde{G_i}$ such that $rk(\widetilde{G_i}+v)=rk(\widetilde{G_i})$.
\end{Case}
\begin{subc}
$rk(\widetilde{G_1}+v)=rk(\widetilde{G_1})$.
\end{subc}
By Lemma \ref{lem1.6} (2), we have $p(\widetilde{G})=p(\widetilde{G_1})+p(\widetilde{G}-\widetilde{G_1})=p(\widetilde{G_1})+p(\widetilde{G_2}+v)$. Hence, $1\leq p(\widetilde{G_1}+v)\leq p(\widetilde{G_2}+v)=1$. It follows that $p(\widetilde{G_1}+v)=p(\widetilde{G_2}+v)=1$.
\begin{subc}
$rk(\widetilde{G_2}+v)=rk(\widetilde{G_2})$.
\end{subc}
By Lemma \ref{lem1.6} (2), $p(\widetilde{G})=p(\widetilde{G_2})+p(\widetilde{G}-\widetilde{G_2})=p(\widetilde{G_2})+p(\widetilde{G_1}+v)$. Hence, $p(\widetilde{G_1}+v)=1$. Furthermore, from $rk(\widetilde{G_2}+v)=rk(\widetilde{G}_2)$, we have $p(\widetilde{G_2}+v)=p(\widetilde{G_2})=1$.
\begin{Case}
$rk(\widetilde{G_i}+v)=rk(\widetilde{G_i})+1$ for each $i=1, 2$.
\end{Case}
Consider the Hermitian adjacency matrix of $\widetilde{G}$:
\begin{align*}
H(\widetilde{G})={\left(
        \begin{array}{ccc}
        H(\widetilde{G_1}) & \tau & 0 \\[3pt]
        \tau^{*} & 0 & \psi^{*} \\[3pt]
        0 & \psi & H(\widetilde{G_2})
        \end{array}
        \right)}.
\end{align*}
The condition $rk(\widetilde{G_1}+v)=rk(\widetilde{G_1})+1$ implies that the equation $H(\widetilde{G_1})X=\tau$ has a solution, say $X_1$. Similarly, suppose $Y_1$ is a solution of the equation $H(\widetilde{G_2})Y=\psi$ and let
\begin{align*}
W={\left(
        \begin{array}{ccc}
        I & -X_1 & 0 \\[3pt]
        0 & 1 & 0 \\[3pt]
        0 & -Y_1 & I
        \end{array}
        \right)}.
\end{align*}
Then
\begin{align*}
W^*H(\widetilde{G})W={\left(
        \begin{array}{ccc}
        H(\widetilde{G_1}) & 0 & 0 \\[3pt]
        0 & -\tau^*X_1-\psi^*Y_1 & 0 \\[3pt]
        0 & 0 & H(\widetilde{G_2})
        \end{array}
        \right)}.
\end{align*}
From $2=p(\widetilde{G})=p(\widetilde{G_1})+p(\widetilde{G_2})$ and Lemma \ref{lem1.4}, we have $-\tau^*X_1-\psi^*Y_1\leq0$. Note that $-\tau^*X_1\neq0$ and $-\psi^*Y_1\neq0$. Hence, at least one of $-\tau^*X_1$ and $-\psi^*Y_1$ is less than 0. If $-\psi^*Y_1>0$, then $-\tau^*X_1<0$. Hence, we have $p(\widetilde{G_1}+v)=p(\widetilde{G_1})=1$. If $-\psi^*Y_1<0$, we have $p(\widetilde{G_2}+v)=p(\widetilde{G_2})=1$. Hence, by the condition of $p(\widetilde{G_1}+v)\leq p(\widetilde{G_2}+v)$, we have $p(\widetilde{G_1}+v)=1$.
\qed

Combining the proof of Cases $1\verb|--|3$ in Claim 3, we have the following observation.
\begin{Obs}\label{obs}
If $rk(\widetilde{G_i}+v)=rk(\widetilde{G_i})+1$ for each $i=1, 2$, and $-\psi^*Y_1>0$, where $\psi$ and $Y_1$ are defined in Case 3 of Claim 3, then $p(\widetilde{G_2}+v)=2$. Otherwise, $p(\widetilde{G_2}+v)=1$.
\end{Obs}

If $rk(\widetilde{G_i}+v)=rk(\widetilde{G_i})$ for each $i=1, 2$, then combining Claim 3 and Observation \ref{obs}, we have $p(\widetilde{G_1}+v)=p(\widetilde{G_2}+v)=1$. In view of Lemma \ref{lem1.9}, we obtain $\widetilde{G_1}+v$ (resp. $\widetilde{G_2}+v$) is switching equivalent to complete multipartite graph or some $\overrightarrow{C_3}(t_1, t_2, t_3)$. Furthermore, if $\widetilde{G_1}+v$ (resp. $\widetilde{G_2}+v$) is switching equivalent to complete multipartite graph, then $v$ is in the partition class containing at least two vertices of $V(G_1+v)$ (resp. $V(G_2+v)$).

If $rk(\widetilde{G_i}+v)=rk(\widetilde{G_i})$ and $rk(\widetilde{G_j}+v)=rk(\widetilde{G_j})+1$ for $i\neq j$. Without loss of generality, assume $rk(\widetilde{G_1}+v)=rk(\widetilde{G_1})$ and $rk(\widetilde{G_2}+v)=rk(\widetilde{G_2})+1$. By Claim 3 and Observation \ref{obs}, we have $p(\widetilde{G_1}+v)=p(\widetilde{G_2}+v)=1$. In view of Lemma \ref{lem1.9}, we obtain $\widetilde{G_1}+v$ is switching equivalent to complete multipartite graph or some $\overrightarrow{C_3}(t_1, t_2, t_3)$. Furthermore, if $\widetilde{G_1}+v$ is switching equivalent to complete multipartite graph, then by $rk(\widetilde{G_1}+v)=rk(\widetilde{G_1})$, we have $v$ is in the partition class containing at least two vertices of $V(G_1+v)$.
Combining $rk(\widetilde{G_2}+v)=rk(\widetilde{G_2})+1$ and Lemma \ref{lem1.9}, we obtain $\widetilde{G_2}+v$ is switching equivalent to some complete $k'$-partite graph satisfying $k'\geq3$ and $v$ is in the partition class containing exactly one vertex of $V(G_2+v)$.

If $rk(\widetilde{G_i}+v)=rk(\widetilde{G_i})+1$ for each $i=1, 2$, and $-\psi^*Y_1<0$, then by Claim 3 and Observation \ref{obs}, we have $p(\widetilde{G_1}+v)=p(\widetilde{G_2}+v)=1$. Combining Lemma \ref{lem1.9} and $rk(\widetilde{G_i}+v)=rk(\widetilde{G_i})+1$ for each $i=1, 2$, we have $\widetilde{G_1}+v$ (resp. $\widetilde{G_2}+v$) is switching equivalent to some complete $k_1$-partite (resp. $k_2$-partite) graph satisfying $k_1\geq3$ (resp. $k_2\geq3$) and $v$ is in the partition class containing exactly one vertex of $V(G_1+v)$ (resp. $V(G_2+v)$).

If $rk(\widetilde{G_i}+v)=rk(\widetilde{G_i})+1$ for each $i=1, 2$, and $-\psi^*Y_1>0$, then by Claim 3 and Observation \ref{obs}, we have $p(\widetilde{G_1}+v)=1$ and $p(\widetilde{G_2}+v)=2$. Combining Lemma \ref{lem1.9} and $rk(\widetilde{G_1}+v)=rk(\widetilde{G_1})+1$, we have $\widetilde{G_1}+v$ is switching equivalent to some complete $(r+1)$-partite graph $K_{q_1, q_2, \ldots, q_r, 1}$ satisfying $r\geq2$ and $v$ is in the partition class containing exactly one vertex. In view of Lemmas \ref{lem1.9} and \ref{lem3.7}, we obtain $\widetilde{G_2}$ is switching equivalent to some complete $k$-partite graph $K_{n'_1, n'_2,\ldots, n'_k}$ and $\widetilde{G_2}+v$ is switching equivalent to some complex unit gain graph $K(0; n_1'', \ldots, n_k''; {a_*}^{i}, {b_*}^{-i}, {c_*}^{1}, {d_*}^{-1}, k-{a_*}-{b_*}-{c_*}-{d_*})$. Thus $\widetilde{G}$ is switching equivalent to some complex unit gain graph $K(q_1,\ldots, q_r; n_1''', \ldots, n_k'''; a^i, b^{-i}, c^{1}, d^{-1}, k-a-b-c-d)$. In the following, we will prove at least two of $a, b, c, d$ are equal to 0. Otherwise, we have $abc\neq0$, $abd\neq0$, $acd\neq0$ or $bcd\neq0$.

Let $H$ denote the Hermitian adjacency matrix of $K(q_1,\ldots, q_r; n_1''', \ldots, n_k'''; a^i, b^{-i}, c^{1}, d^{-1}, k-a-b-c-d)$. Suppose $abc\neq0$ and note that $r\geq2$, then
\begin{align*}
H_1={\left(
        \begin{array}{cccccc}
        0 & 1 & 1 & 0 & 0 & 0 \\[3pt]
        1 & 0 & 1 & 0 & 0 & 0 \\[3pt]
        1 & 1 & 0 & i & -i & 1 \\[3pt]
        0 & 0 & -i & 0 & 1 & 1 \\[3pt]
        0 & 0 & i & 1 & 0 & 1 \\[3pt]
        0 & 0 & 1 & 1 & 1 & 0 \\[3pt]
        \end{array}
        \right)}
\end{align*}
is a submatrix of $H$. By a direct calculation, we have $p(H_1)=3$. Thus, $p(\widetilde{G})=p(H)\geq p(H_1)=3$, a contradiction.
Similarly, we also have $abd=0, acd=0$  and $bcd=0$. Thus, we have at least two of $a, b, c, d$ are equal to 0.

If $d=0$, then $K(q_1,\ldots, q_r; n_1''', \ldots, n_k'''; a^i, b^{-i}, c^{1}, 0^{-1}, k-a-b-c)$ is a mixed graph. Hence, we have $\widetilde{G}$ is switching equivalent to mixed graph $K(q_1,\ldots, q_r; n_1''', \ldots, n_k'''; a^i, b^{-i}, c^{1}, 0^{-1}, k-a-b-c)$.

If $d>0$, we proceed by considering four cases.

{\bf{Case 1.}}\ $a, b, c$ are all equal to 0.

In this case, there exists a diagonal matrix $D_1$ with $(D_1)_{uu}=-1$ if $u\in V(G_1+v)$. Otherwise, $(D_1)_{uu}=1$. It is easy to see that $D_1^{-1}HD_1$ is the adjacency matrix of some simple graph $K(q_1,\ldots, q_r; n_1''', \ldots, n_k'''; d)$. Thus $\widetilde{G}$ is switching equivalent to $K(q_1,\ldots, q_r; n_1''', \ldots, n_k'''; d)$.

{\bf{Case 2.}}\ $a>0$ and $b=c=0$.

In this case, there exists a diagonal matrix $D_2$ with $(D_2)_{uu}=i$ if $u\in V(G_1+v)$. Otherwise, $(D_2)_{uu}=1$. It is easy to see that $D_2^{-1}HD_2$ is the Hermitian adjacency matrix of some mixed graph
$$K(q_1,\ldots, q_r; n_1, \ldots, n_k; d^i, 0^{-i}, a^{1}, 0^{-1}, k-a-d).$$
Thus $\widetilde{G}$ is switching equivalent to some mixed graph
$$K(q_1,\ldots, q_r; n_1, \ldots, n_k; d^i, 0^{-i}, a^{1}, 0^{-1}, k-a-d).$$

{\bf{Case 3.}}\ $b>0$ and $a=c=0$.

In this case, there exists a diagonal matrix $D_3$ with $(D_3)_{uu}=-i$ if $u\in V(G_1+v)$. Otherwise, $(D_3)_{uu}=1$. It is easy to see that $D_3^{-1}HD_3$ is the Hermitian adjacency matrix of some mixed graph
$$K(q_1,\ldots, q_r; n_1, \ldots, n_k; 0^i, d^{-i}, b^{1}, 0^{-1}, k-b-d).$$
Thus $\widetilde{G}$ is switching equivalent to some mixed graph
$$K(q_1,\ldots, q_r; n_1, \ldots, n_k; 0^i, d^{-i}, b^{1}, 0^{-1}, k-b-d).$$

{\bf{Case 4.}}\ $c>0$ and $a=b=0$.

In this case, there exists a diagonal matrix $D_4$ with $(D_4)_{uu}=-i$ if $u\in V(G_1+v)$. Otherwise, $(D_4)_{uu}=1$. It is easy to see that $D_4^{-1}HD_4$ is the Hermitian adjacency matrix of some mixed graph
$$K(q_1,\ldots, q_r; n_1, \ldots, n_k; c^i, d^{-i}, 0^{1}, 0^{-1}, k-c-d).$$
Thus $\widetilde{G}$ is switching equivalent to some mixed graph
$$K(q_1,\ldots, q_r; n_1, \ldots, n_k; c^i, d^{-i}, 0^{1}, 0^{-1}, k-c-d).$$

Since simple graphs are mixed graphs, according to the above analysis, without loss of generality, assume that $\widetilde{G}$ is switching equivalent to some mixed graph $K(q_1,\ldots, q_r; n_1, \ldots, n_k; a^i, b^{-i}, c^{1}, 0^{-1}, k-a-b-c)$, where $a, b, c$ are nonnegative integers such that at least one of them is equal to 0 and $1\leq a+b+c\leq k$.

Note that $[{v}, V(G_2)]$ is an edge cut of $K(q_1,\ldots, q_r; n_1, \ldots, n_k; a^i, b^{-i}, c^{1}, 0^{-1}, k-a-b-c)$. We can perform a sequence of two-way switching and operations of taking the converse so that $\widetilde{G}$ is switching equivalent to one of the following mixed graphs:
\begin{wst}
  \item[{\rm (i)}] $K(q_1,\ldots, q_r; n_1, \ldots, n_k; p)$;
  \item[{\rm (ii)}] $K(q_1,\ldots, q_r; n_1, \ldots, n_k; a^i, b^{-i}, 0^{1}, 0^{-1}, k-a-b)$, where $a\geq b\geq1$;
  \item[{\rm (iii)}] $K(q_1,\ldots, q_r; n_1, \ldots, n_k; a^i, 0^{-i}, c^{1}, 0^{-1}, k-a-c)$, where $a\geq c\geq1$.
\end{wst}

If $\widetilde{G}$ is switching equivalent to $K(q_1,\ldots, q_r; n_1, \ldots, n_k; p)$, then according to Lemmas \ref{twin}, \ref{add} and Corollary \ref{cor}, $T_{\widetilde{G}}$ is switching equivalent to $K(r; k; p)$, where $r\geq2, k\geq2, p\geq1$, satisfying one of the following conditions:
\begin{wst}
  \item[{\rm (1)}] $p=1$;
  \item[{\rm (2)}] $p\geq2$ and $k-p\leq1$;
  \item[{\rm (3)}] $p\geq2, k-p\geq2$ and $\frac{1}{r}+\frac{1}{p}+\frac{1}{k-p-1}\geq1$.
\end{wst}
By the definition of twin reduction graph and $G$ without pendant vertices, it is easy to see that $\widetilde{G}$ is switching equivalent to $K(q_1,\ldots, q_r; n_1, \ldots, n_k; p)$, where $r\geq2, k\geq2, p\geq1$, satisfying one of the conditions stated in Theorem \ref{thm1.2} ``only if " part (ii).

If $\widetilde{G}$ is switching equivalent to $K(q_1,\ldots, q_r; n_1, \ldots, n_k; a^i, b^{-i}, 0^{1}, 0^{-1}, k-a-b)$, where $a\geq b\geq1$, then according to Lemmas \ref{twin}, \ref{add} and \ref{3.8}, $T_{\widetilde{G}}$ is switching equivalent to $K(r; k; a^i, b^{-i}, 0^{1}, 0^{-1}, s)$, where $r\geq2, k\geq2, a\geq b\geq1, s=k-a-b$, satisfying $a=b=1$ and $r=2$. By the definition of twin reduction graph and $G$ without pendant vertices, it is easy to see that $\widetilde{G}$ is switching equivalent to
$$K(q_1,\ldots, q_r; n_1, \ldots, n_k; a^i, b^{-i}, 0^{1}, 0^{-1}, s)$$
where $r\geq2, k\geq2, a\geq b\geq1, s=k-a-b,$
satisfying $a=b=1$ and $r=2$.

If $\widetilde{G}$ is switching equivalent to $K(q_1,\ldots, q_r; n_1, \ldots, n_k; a^i, 0^{-i}, c^{1}, 0^{-1}, k-a-c)$, where $a\geq c\geq1$, then according to Lemmas \ref{twin}, \ref{add} and \ref{3.10}, $T_{\widetilde{G}}$ is switching equivalent to $K(r; k; a^i, 0^{-i}, c^{1}, 0^{-1}, s)$, with $r\geq2, k\geq2, a\geq c\geq1,  s=k-a-c$, satisfying one of the following conditions:
\begin{wst}
  \item[{\rm (1)}] $a=c=1$ and either $s=0$ or $s=1$;
  \item[{\rm (2)}] $a=c=1, s=2$ and either $r=3$ or $r=4$;
  \item[{\rm (3)}] $a=c=1$, $s=3$ and $r=3$;
  \item[{\rm (4)}] $a=c=1$, $s\geq2$ and $r=2$;
  \item[{\rm (5)}] $a=c=2, s=0$ and $2\leq r\leq4$;
  \item[{\rm (6)}] $a=c=2, s=1$ and $r=2$;
  \item[{\rm (7)}] $a=3, s=0$ and $r=c=2$;
  \item[{\rm (8)}] $a=4, c=r=2$ and $s=0$;
  \item[{\rm (9)}] $a\geq2, c=1$ and $\frac{as-1}{a+s}\leq\frac{1}{r-1}$.
\end{wst}
By the definition of twin reduction graph and $G$ without pendant vertices, it is easy to see that $\widetilde{G}$ is switching equivalent to
$$K(q_1,\ldots, q_r; n_1, \ldots, n_k; a^i, 0^{-i}, c^{1}, 0^{-1}, s)$$
with $r\geq2, k\geq2, a\geq c\geq1, s=k-a-c$, satisfying one of the conditions stated in Theorem \ref{thm1.2} ``only if " part (iv).
This completes the proof.
\qed
\section*{\normalsize Acknowledgement}
This work was supported by  NSFC (Grant Nos. 11871479, 12071484), Hunan Provincial Natural Science Foundation (Grant Nos. 2018JJ2479, 2020JJ 4675).

{
}

\end{document}